%% file: arXiv-maxEnt.tex
\documentclass[a4paper,11pt,english]{article}
\usepackage{arxiv}
\input{arXiv-preamble}

\subimport{macros/}{all-macros.tex}

\begin{document}

    \title{\input{info/title.tex}{}}
    \author{Alex Goessmann*}
    \author{Martin Eigel}
    \affil{Weierstrass Institute of Applied Analysis and Stochastics \\
    Anton-Wilhelm-Amo-Straße 39\\
    Berlin, 10117, Germany\\ 
    \ \\
    *alex.goessmann@wias-berlin.de
    }

    \maketitle

    \begin{abstract}
        \input{info/abstract}
    \end{abstract}

    \subimport{}{./sections/introduction.tex}
    \subimport{}{./sections/mean-polytope.tex}
    \subimport{}{./sections/exp-family.tex}
    \subimport{}{./sections/gen-rep.tex}
    \subimport{}{./sections/boolean-stat.tex}
    \subimport{}{./sections/outlook.tex}

    \vskip 0.2in
    \bibliographystyle{plainnat}
    \bibliography{references.bib}

    \appendix

\end{document}

%% file: arXiv-preamble.tex
\usepackage{amsmath,amsfonts,amssymb,amsthm,enumerate,times}
\usepackage{mathtools}
\usepackage[usenames,dvipsnames]{color}
\usepackage{hyperref}
\hypersetup{
    breaklinks,
    colorlinks,
    linkcolor=gray,
    citecolor=gray,
    urlcolor=gray,
    pdftitle={},
    pdfauthor={Alex Goessmann}
}

\usepackage{tikz}
\usepackage{graphicx}
\usepackage{float}
\usepackage{comment}
\usepackage{csquotes}
\usepackage{braket}
\usepackage{listings}
\usepackage{verbatim}
\usepackage{etoolbox}
\usepackage[utf8]{inputenc}
\usepackage[english]{babel}
\usepackage[T1]{fontenc}
\usepackage{titlesec}
\usepackage{fancyhdr}
\usepackage{bbm}
\usepackage{bm}
\usepackage{algpseudocode}
\usepackage{algorithm}
\usepackage{lipsum}

\usepackage{minted}
\usepackage{tikz-3dplot}
\usepackage{import}
\usepackage{authblk}
\usepackage{catchfile}

\newtheorem{theorem}{Theorem}
\newtheorem{lemma}{Lemma}

\newtheorem{definition}{Definition}
\newtheorem{example}{Example}

\DeclareUnicodeCharacter{FB01}{fi}
\usepackage[round]{natbib}
\usepackage{wasysym}

%% file: macros/all-macros.tex
\subimport{}{general_macros.tex}

\subimport{}{tc_macros.tex}

\subimport{}{tikz_macros.tex}

\subimport{}{organization_macros.tex}

\input{tensor-notation/variables-macros}

\input{tensor-notation/cores-macros}

\input{tensor-notation/operations-macros}

\input{tensor-notation/graph-macros}

\input{chapter-specific/logic-macros}

\input{chapter-specific/probability-macros}

\input{chapter-specific/quantum-macros}

\input{chapter-specific/causality-macros}

\input{chapter-specific/polytopes-macros}

%% file: macros/general_macros.tex

\newcommand{\onetuple}[1]{1_{#1}}

\newcommand{\ozset}{\{0,1\}}

\newcommand{\uniquantwrtof}[2]{\forall{#1}:{#2}}

{\begin{center}
     \begin{algorithmic}
         \hspace{1cm}}
{\end{algorithmic}\end{center}}

\newcommand{\defcols}{\,:\,} 
\newcommand{\wcols}{\,:\,} 
\newcommand{\ncond}{,\,} 

\newcommand{\andspace}{\quad\text{and}\quad}
\newcommand{\ifspace}{\text{if}\quad}
\newcommand{\elsetext}{\text{else}}
\newcommand{\stspace}{\quad \text{subject to} \quad}

\newcommand{\entropysymbol}{\mathbb{H}}

\newcommand{\sentropyofwrt}[2]{\entropysymbol_{#2}\left[#1\right]}

\newcommand{\centropyofwrt}[3]{\entropysymbol_{#3}\left[#1,#2\right]}

\newcommand{\rr}{\mathbb{R}}
\newcommand{\nn}{\mathbb{N}}

\newcommand{\interiorof}[1]{{#1}^{\circ}}
\newcommand{\sbinteriorof}[1]{{\left(#1\right)}^{\circ}}

\newcommand{\cardof}[1]{\left|#1\right|}

\newcommand{\imageof}[1]{\mathrm{im}\left(#1\right)}

\newcommand{\convhullof}[1]{\mathrm{conv}\left(#1\right)}
\newcommand{\cubeof}[1]{[0,1]^{#1}}

\newcommand{\argmax}{\mathrm{argmax}}


\newcommand{\expof}[1]{\mathrm{exp}\left[#1\right]}
\newcommand{\probtensor}{\mathbb{P}}
\newcommand{\probtensorof}[1]{\probtensor^{#1}}

\newcommand{\secprobtensor}{\tilde{\mathbb{P}}}
\newcommand{\secprobat}[1]{\secprobtensor[#1]}

\newcommand{\probat}[1]{\probtensor\left[#1\right]}

\newcommand{\probwith}{\probat{\shortcatvariables}}

\newcommand{\expectationof}[1]{\mathbb{E}\left[#1\right]}

\newcommand{\lnof}[1]{\ln \left[ #1 \right] }

\newcommand{\ones}{\mathbb{I}}

\newcommand{\onesat}[1]{\ones\left[#1\right]}

\newcommand{\zeros}{0}
\newcommand{\zerosat}[1]{\zeros\left[#1\right]}


\newcommand{\restrictionofto}[2]{{#1}|_{#2}}





\newcommand{\formulaset}{\mathcal{F}}
\newcommand{\formulasetof}[1]{\formulaset_{#1}}








\newcommand{\parspace}{\rr^{\seldim}}
\newcommand{\fullparcube}{[0,1]^{\seldim}}




\newcommand{\mlnatomsymbol}{[\catorder]}

\newcommand{\partitionfunction}{\mathcal{Z}}










\newcommand{\indexinterpretation}{I}
\newcommand{\indexinterpretationof}[1]{\indexinterpretation_{#1}}

\newcommand{\indexinterpretationofat}[2]{\indexinterpretationof{#1}\left[#2\right]}

\newcommand{\invindexinterpretationof}[1]{\indexinterpretation_{#1}^{-1}}

\newcommand{\invindexinterpretationofat}[2]{\invindexinterpretationof{#1}(#2)}




\newcommand{\arbset}{\mathcal{U}}




%% file: macros/tc_macros.tex
\newcommand{\coloredmatrixof}[2]{\begin{pmatrix}
#1
\end{pmatrix}\left[#2\right]}


\newcommand{\bencodingof}[1]{\beta^{#1}}

\NewDocumentCommand{\bencodingofat}{m m o}{
    \IfValueTF{#3}{%
        \bencodingof{#1}\left[#2\middle|#3\right]
    }{%
        \bencodingof{#1}\left[#2\right]
    }%
}

\newcommand{\sencodingof}[1]{\sigma^{#1}}



\newcommand{\expfamilyof}[1]{\Gamma^{#1}}
\newcommand{\expfamily}{\genexpfamily}
\newcommand{\genexpfamily}{\expfamilyof{\sstat,\basemeasure}}

\newcommand{\realizabledistsof}[1]{\Lambda^{#1}}

\newcommand{\cansof}[1]{\Lambda^{#1}}



\newcommand{\decvariable}{I}
\newcommand{\decvariableof}[1]{\decvariable_{#1}}
\newcommand{\decindex}{i} 

\newcommand{\decdim}{n}

\newcommand{\decindexin}{\decindex\in[\decdim]}
\newcommand{\indexeddecvariable}{\decvariable=\decindex}


\newcommand{\cprankof}[1]{\mathrm{rank}\left(#1\right)}


%% file: macros/tikz_macros.tex
\newcommand{\dotsize}{0.15cm}
\newcommand{\nodeminsize}{0.8cm}
\newcommand{\nodegrayscale}{gray!50}
\newcommand{\colorlabelsize}{\tiny}
\newcommand{\corelabelsize}{\small}
\newcommand{\variablenodelabelsize}{\small}

\newcommand{\probcolor}{purple!60}
\newcommand{\concolor}{cyan}


\usetikzlibrary{arrows.meta}
\usetikzlibrary{shapes,positioning}
\usetikzlibrary{decorations.markings}
\usetikzlibrary{calc}
\usetikzlibrary{matrix}

\tikzset{
    midarrow/.style={
        postaction={decorate},
        decoration={markings, mark=at position 0.5 with {\arrow{>}}}
    },
    midbackarrow/.style={
        postaction={decorate},
        decoration={markings, mark=at position 0.5 with {\arrow{<}}}
    },
    ->-/.style={midarrow},
    -<-/.style={midbackarrow}
}

\newcommand{\drawtextnode}[3]{
    \node[anchor=center,align=center] at (#1,#2) {#3};
}


\newcommand{\drawvariablenode}[4]{
    \node [circle, draw, thick, fill=\nodegrayscale, minimum size = \nodeminsize] (#4) at (#1,#2) {};
    \drawtextnode{#1}{#2}{\variablenodelabelsize #3}
}

\newcommand{\drawtensorblock}[4]{
    \draw (#1-#3,#2-1) rectangle (#1+#3,#2+1);
    \drawtextnode{#1}{#2}{\corelabelsize #4}
}

\newcommand{\shortminus}{\scalebox{0.4}[1.0]{$-$}}

\newcommand{\drawvariabledot}[2]{
    \draw[fill] (#1,#2) circle (\dotsize);
}

\newcommand{\drawhwire}[3]{
    \draw[] (#1,#2) -- (#1+2,#2) node[midway,above] {\colorlabelsize #3};
}
\newcommand{\drawvwire}[3]{
    \draw[] (#1,#2) -- (#1,#2-2) node[midway,right] {\colorlabelsize #3};
}

\newcommand{\drawtdwire}[3]{
    \draw[->-] (#1,#2-2) -- (#1,#2) node[midway,right] {\colorlabelsize #3};
}
\newcommand{\drawddwire}[3]{
    \draw[->-] (#1,#2+2) -- (#1,#2) node[midway,right] {\colorlabelsize #3};
}


\newcommand{\drawvectormark}[2]{
    \draw[black, line width=0.6pt] (#1,#2) ++(-0.24,0.24) -- ++(0.48,-0.48);
    \draw[black, line width=0.6pt] (#1,#2) ++(-0.24,-0.24) -- ++(0.48,0.48);
}




\newcommand{\stantikzpic}[1]{
    \begin{center}
    \begin{tikzpicture}[scale=0.35, thick]
    #1
    \end{tikzpicture}
    \end{center}
}


\tikzset{
    sudokuspec/.style={
        sudokucell/.style={
            minimum size=0.6cm,
            draw=gray!50,
            anchor=center
        },
        sudokumatrix/.style={
            matrix of nodes,
            nodes=sudokucell,
            inner sep=0pt,
            row sep=-\pgflinewidth,
            column sep=-\pgflinewidth,
            draw=black,
            very thick,
        }
    }
}

%% file: macros/organization_macros.tex






\newcommand{\ComputationActivationNetwork}{Computation-Activation Network}
\newcommand{\ComputationActivationNetworks}{Computation-Activation Networks}

\newcommand{\CompActNet}{CompActNet}
\newcommand{\CompActNets}{CompActNets}



\newcommand{\HybridLogicNetworks}{Hybrid Logic Networks}










\newcommand{\defref}[1]{Def.~\ref{#1}}
\newcommand{\theref}[1]{Thm.~\ref{#1}}
\newcommand{\lemref}[1]{Lem.~\ref{#1}}

\newcommand{\problemref}[1]{Problem~\ref{#1}}

\newcommand{\exaref}[1]{Example~\ref{#1}}

\newcommand{\secref}[1]{Sect.~\ref{#1}}
\newcommand{\figref}[1]{Figure~\ref{#1}}

%% file: macros/tensor-notation/variables-macros.tex



\newcommand{\selvariable}{L}
\newcommand{\seldim}{p}
\newcommand{\selindex}{\ell}
\newcommand{\selenumerator}{s}
\newcommand{\selorder}{n}

\newcommand{\catvariable}{X}
\newcommand{\catdim}{m}
\newcommand{\catindex}{x} 
\newcommand{\catenumerator}{\atomenumerator}
\newcommand{\catorder}{\atomorder}

\newcommand{\headvariable}{Y} 
\newcommand{\headdim}{n}
\newcommand{\headindex}{y}



\newcommand{\selvariableof}[1]{\selvariable_{#1}}
\newcommand{\catvariableof}[1]{\catvariable_{#1}}
\newcommand{\headvariableof}[1]{\headvariable_{#1}}

\newcommand{\shortcatvariablelist}{\catvariableof{[\atomorder]}}

\newcommand{\shortcatindices}{\catindexof{[\catorder]}}

\newcommand{\seldimof}[1]{\seldim_{#1}}
\newcommand{\catdimof}[1]{\catdim_{#1}}
\newcommand{\headdimof}[1]{\headdim_{#1}}

\newcommand{\selindexof}[1]{\selindex_{#1}}
\newcommand{\catindexof}[1]{\catindex_{#1}}
\newcommand{\headindexof}[1]{\headindex_{#1}}

\newcommand{\selindexin}{\selindex\in[\seldim]}

\newcommand{\selenumeratorin}{\selenumerator\in[\selorder]}
\newcommand{\catenumeratorin}{\catenumerator\in[\catorder]}

\newcommand{\indexedcatvariableof}[1]{\catvariableof{#1}=\catindexof{#1}}
\newcommand{\indexedselvariableof}[1]{\selvariableof{#1}=\selindexof{#1}}
\newcommand{\indexedheadvariableof}[1]{\headvariableof{#1}=\headindexof{#1}}

\newcommand{\indexedselvariable}{\indexedselvariableof{}}

\newcommand{\catstatesof}[1]{[\catdimof{#1}]}

\newcommand{\shortcatvariables}{\shortcatvariablelist}
\newcommand{\indexedshortcatvariables}{\shortcatvariables=\shortcatindices}
\newcommand{\shortcatindicesin}{\shortcatindices\in\facstates}

\newcommand{\headvariables}{\headvariableof{[\seldim]}}

\newcommand{\secselindex}{\tilde{\selindex}}

\newcommand{\facstates}{\bigtimes_{\atomenumeratorin}\catstatesof{\atomenumerator}}


\newcommand{\seccatorder}{p} 


\newcommand{\tildecatindex}{\tilde{\catindex}}
\newcommand{\tildecatindexof}[1]{\tildecatindex_{#1}}

%% file: macros/tensor-notation/cores-macros.tex
\newcommand{\softactsymbol}{\alpha}
\newcommand{\softactsymbolof}[1]{\softactsymbol^{#1}}
\newcommand{\softactsymbolofat}[2]{\softactsymbol^{#1}\left[#2\right]}
\newcommand{\softacttensor}{\softactsymbol^\weightparam}
\newcommand{\softacttensorat}[1]{\softacttensor\left[#1\right]}
\newcommand{\softacttensorwith}{\softacttensorat{\headvariables}}
\newcommand{\softactleg}{\softactsymbolof{\selindex,\weightparam}}

\newcommand{\softactlegwith}{\softactsymbolofat{\selindex,\weightparam}{\headvariableof{\selindex}}}


\newcommand{\onehotmap}{\epsilon}
\newcommand{\onehotmapof}[1]{\onehotmap_{#1}}
\newcommand{\onehotmapofat}[2]{\onehotmap_{#1}\left[#2\right]}

\newcommand{\tbasis}{\onehotmapof{1}}
\newcommand{\tbasisat}[1]{\tbasis\left[#1\right]}
\newcommand{\fbasis}{\onehotmapof{0}}
\newcommand{\fbasisat}[1]{\fbasis\left[#1\right]}


\newcommand{\canparam}{\weightparam} 
\newcommand{\canparamof}[1]{\canparam_{#1}}
\newcommand{\canparamat}[1]{\canparam\left[#1\right]}
\newcommand{\canparamofat}[2]{\canparamof{#1}\left[#2\right]}
\newcommand{\canparamwith}{\canparamat{\selvariable}}

\newcommand{\canparamwithin}{\canparamwith\in\parspace}


\newcommand{\hardactsymbol}{\kappa}
\newcommand{\hardactsymbolof}[1]{\hardactsymbol^{#1}}
\newcommand{\hardactsymbolofat}[2]{\hardactsymbol^{#1}\left[#2\right]}
\newcommand{\hardacttensor}{\hardactsymbolof{\hardparam}}
\newcommand{\hardacttensorwith}{\hardacttensor\left[\headvariables\right]}

\newcommand{\hardactlegwith}{\hardactsymbolofat{\selindex,\hardparam}{\headvariableof{\selindex}}}

\newcommand{\kcoreof}[1]{\hardactsymbolof{#1}}

\newcommand{\kcoreofat}[2]{\hardactsymbolofat{#1}{#2}}


\newcommand{\meanparam}{\mu}
\newcommand{\meanparamof}[1]{\meanparam_{#1}}
\newcommand{\meanparamat}[1]{\meanparam\left[#1\right]}
\newcommand{\meanparamofat}[2]{\meanparamof{#1}\left[#2\right]}

\newcommand{\meanparamwith}{\meanparamat{\selvariable}}

\newcommand{\basemeasure}{\nu}
\newcommand{\basemeasureof}[1]{\basemeasure^{#1}}
\newcommand{\basemeasureofat}[2]{\basemeasure^{#1}\left[#2\right]}
\newcommand{\basemeasureat}[1]{\basemeasure\left[#1\right]}
\newcommand{\basemeasurewith}{\basemeasureat{\shortcatvariables}}

\newcommand{\acttensor}{\xi} 
\newcommand{\acttensorof}[1]{\acttensor^{#1}}

\newcommand{\acttensorat}[1]{\acttensor\left[#1\right]}
\newcommand{\acttensorwith}{\acttensorat{\headvariables}}
\newcommand{\acttensorleg}{\acttensorof{\selindex}}

\newcommand{\acttensorlegwith}{\acttensorleg\left[\headvariableof{\selindex}\right]}




\newcommand{\tnet}{\tau}
\newcommand{\tnetof}[1]{\tnet^{#1}}

\newcommand{\extnet}{\tnetof{\graph}}





\newcommand{\hypercore}{\tau}
\newcommand{\hypercoreat}[1]{\hypercore\left[#1\right]}
\newcommand{\hypercorewith}{\hypercoreat{\shortcatvariables}}

\newcommand{\hypercoreof}[1]{\hypercore^{#1}}
\newcommand{\hypercoreofat}[2]{\hypercoreof{#1}\left[#2\right]}

%% file: macros/tensor-notation/operations-macros.tex
\newcommand{\contractionof}[2]{\left\langle #1\right\rangle_{\left[ #2 \right]}}

\newcommand{\contraction}[1]{\contractionof{#1}{\varnothing}}

%% file: macros/tensor-notation/graph-macros.tex
\newcommand{\graph}{\mathcal{G}}

\newcommand{\nodes}{\mathcal{V}}



\newcommand{\node}{v}

\newcommand{\nodein}{\node\in\nodes}

\newcommand{\edges}{\mathcal{E}}

\newcommand{\edge}{e}

\newcommand{\edgein}{\edge\in\edges}


\newcommand{\elformat}{\mathrm{EL}}
\newcommand{\cpformat}{\mathrm{CP}}

\newcommand{\ttformat}{\mathrm{TT}}

\newcommand{\elgraph}{\elformat}

\newcommand{\tnset}{\mathcal{T}}
\newcommand{\tnsetof}[1]{\tnset^{#1}} 


%% file: macros/chapter-specific/logic-macros.tex



\newcommand{\atomformulaset}{\formulasetof{\mlnatomsymbol}}

\newcommand{\mintermformulaset}{\mathcal{F}_{\land}}

\newcommand{\atomorder}{d}

\newcommand{\atomenumerator}{k}

\newcommand{\atomenumeratorin}{\atomenumerator\in[\atomorder]}







\newcommand{\formula}{f}
\newcommand{\formulaof}[1]{\formula_{#1}}
\newcommand{\formulaat}[1]{\formula\left[#1\right]}

\newcommand{\formulaofat}[2]{\formulaof{#1}\left[#2\right]}




\newcommand{\enumformula}{\formulaof{\selindex}}









%% file: macros/chapter-specific/probability-macros.tex


\newcommand{\sstat}{t}
\newcommand{\sstatat}[1]{\sstat\left(#1\right)}

\newcommand{\hlnstat}{\sstat}

\newcommand{\sstatcoordinateof}[1]{\sstat_{#1}}

\newcommand{\sstatcoordinateofat}[2]{\sstat_{#1}\left[#2\right]}

\newcommand{\sstatccwith}{\bencodingofat{\sstat}{\headvariables,\shortcatvariables}} 

\newcommand{\sencsstat}{\sencodingof{\sstat}}
\newcommand{\sencsstatat}[1]{\sencodingof{\sstat}\left[#1\right]}
\newcommand{\sencsstatwith}{\sencsstatat{\shortcatvariables,\selvariable}}

\newcommand{\bencsstatat}[1]{\bencodingof{\sstat}\left[#1\right]}
\newcommand{\bencsstatwith}{\bencsstatat{\headvariables,\shortcatvariables}}

\newcommand{\expdistof}[1]{\probtensorof{#1}}
\newcommand{\expdistofat}[2]{\expdistof{#1}[#2]}
\newcommand{\expdist}{\probtensorof{(\sstat,\canparam,\basemeasure)}}

\newcommand{\expdistat}[1]{\expdist\left[#1\right]}

\newcommand{\maxentfamilyof}[2]{\Gamma^{#1,#2}}


\newcommand{\trivbm}{\ones}

%% file: macros/chapter-specific/quantum-macros.tex











%% file: macros/chapter-specific/causality-macros.tex


%% file: macros/chapter-specific/polytopes-macros.tex
\newcommand{\distmeasuredby}[1]{\Lambda^{#1}}

\newcommand{\stateset}{\mathcal{X}}
\newcommand{\sstatencoding}{\sencodingof{\sstat}}

\newcommand{\genstatshortcatencoding}[1]{\gamma^{\sstat}\left[\indexedshortcatvariables,\selvariable\right]}


\newcommand{\meanset}{\mathcal{M}}
\newcommand{\meansetof}[1]{\meanset_{#1}}
\newcommand{\genmeanset}{\meanset_{\sstat,\basemeasure}}
\newcommand{\hlnmeanset}{\meanset_{\hlnstat,\ones}}

\newcommand{\imset}{\mathcal{N}}
\newcommand{\imsetof}[2]{\imset_{#1}^{#2}}
\newcommand{\genimset}{\imsetof{\sstat,\basemeasure}{}}
\newcommand{\genfaceimset}{\imsetof{\sstat,\basemeasure}{\facesymbol}}

\newcommand{\imelement}{v}
\newcommand{\imelementof}[1]{\imelement^{#1}}
\newcommand{\imelementat}[1]{\imelement\left[#1\right]}
\newcommand{\imelementofat}[2]{\imelement^{#1}\left[#2\right]}
\newcommand{\imelementwith}{\imelementat{\selvariable}}


\newcommand{\facesymbol}{\mathcal{F}} 
\newcommand{\facesymbolof}[1]{\facesymbol^{#1}}
\newcommand{\faceto}[1]{\facesymbol(#1)}
\newcommand{\facelatticeof}[1]{L\left(#1\right)}
\newcommand{\genfacelattice}{\facelatticeof{\genmeanset}}
\newcommand{\facein}{\facesymbol\in\genfacelattice}

\newcommand{\secfacesymbol}{\tilde{\facesymbol}}



\newcommand{\genfacemeasure}{\basemeasureof{\sstat,\facesymbol}}
\newcommand{\genfacemeasureat}[1]{\basemeasureofat{\sstat,\facesymbol}{#1}}
\newcommand{\genfacemeasurewith}{\genfacemeasureat{\shortcatvariables}}


\newcommand{\genmean}{\meanparam^*}
\newcommand{\genmeanat}[1]{\genmean[#1]}
\newcommand{\genmeanwith}{\genmeanat{\selvariable}}

\newcommand{\canparamsetof}[1]{\mathcal{P}_{#1}}
\newcommand{\gencanparamset}{\canparamsetof{\sstat,\basemeasure}}

\newcommand{\weightparam}{\theta}

\newcommand{\weightparamat}[1]{\weightparam\left[#1\right]}

\newcommand{\weightparamwith}{\weightparamat{\selvariable}}
\newcommand{\weightparamin}{\weightparam\in\weightspace}
\newcommand{\weightspace}{\rr^{\seldim}}

\newcommand{\cumfunctionwrt}[1]{A^{#1}}

\newcommand{\cumfunction}{\cumfunctionwrt{(\sstat,\basemeasure)}}
\newcommand{\cumfunctionof}[1]{\cumfunction(#1)}

\newcommand{\forwardmapwrt}[1]{F^{#1}}
\newcommand{\forwardmap}{\forwardmapwrt{(\sstat,\basemeasure)}}
\newcommand{\forwardmapwrtof}[2]{\forwardmapwrt{#1}(#2)}
\newcommand{\forwardmapof}[1]{\forwardmapwrtof{(\sstat,\basemeasure)}{#1}}

\newcommand{\backwardmapwrt}[1]{B^{#1}}
\newcommand{\backwardmap}{\backwardmapwrt{(\sstat,\basemeasure)}}
\newcommand{\backwardmapwrtof}[2]{\backwardmapwrt{#1}(#2)}
\newcommand{\backwardmapof}[1]{\backwardmapwrtof{(\sstat,\basemeasure)}{#1}}

\newcommand{\hardlegset}{\mathcal{U}}
\newcommand{\hardparam}{(\hardlegset,\headindexof{\hardlegset})}


%% file: info/title.tex
Tensor network representations of discrete maximum entropy distributions via mean polytopes

%% file: info/abstract.tex
We present tensor network representations for discrete maximum entropy distributions under expectation constraints.
To this end, we introduce \ComputationActivationNetworks{} (\CompActNets{}), a tensor network architecture that subsumes exponential families.
By leveraging the geometry of the convex polytope of realizable expectation vectors, we represent any maximum entropy distribution in the same architecture.
We exploit the fact that proper faces of this polytope correspond to the boundary closure of exponential families, which restricts the distribution's support.
We then derive explicit representations for the support within the \CompActNet{} architecture.
The proposed framework suggests tensor network ranks as complexity measures for faces.
Finally, a case study on Boolean statistics links the geometry of 0/1-polytopes directly to propositional formulas.



%% file: sections/introduction.tex
\section{Introduction}

The Shannon entropy of a distribution quantifies the amount of structure, or information, and therefore appears as an important concept across many scientific areas.
By Shannon's source code theorem \cite{shannon_mathematical_1948}, entropy describes fundamental limits of information transfer.
In statistical physics, founded on the ergodic assumptions, maximum entropy distributions such as Maxwell-Boltzmann distributions are the gold standard to model equilibrium distributions \cite{uffink_compendium_2007,presse_principles_2013}.

The principle of entropy maximization also appears naturally in learning problems \cite{murphy_probabilistic_2023}.
Extracting characteristics observed in data does not determine an estimated model yet.
Choosing the distribution with the least structure among those reflecting the observed characteristics is a regularization principle.
Here one can choose the Shannon entropy to quantify the negative amount of structure in the distribution and maximize the quantity with respect to the constraints.
These maximum entropy problems are the dual problems to convex maximum likelihood problems \cite{koller_probabilistic_2009}.
Maximum entropy principles appear further in reinforcement learning \cite{yoon_maximum_2024,song_survey_2025}, AGI \cite{miotto_new_2025} and Bayesian NN \cite{rathnakumar_bayesian_2026}. 

A classical result states that maximum entropy distributions belong to exponential families, if and only if the corresponding mean parameter is in the relative interior of the mean polytope \cite{barndorff-nielsen_information_1978}.
A natural extension is therefore the closure of the exponential families as investigated in \cite{csiszar_closures_2005,malago_note_2010}.
Alternative methods use toric ideals to allow for infinite weight vectors \cite{sottile_toric_2008}.
Numerical stability of the maximum entropy distributions near the boundary is studied in \cite{straszak_computing_2017,straszak_maximum_2019}.

Less obvious, studying general maximum entropy distributions might be a key step towards a mathematically sound unification of logical and probabilistic models.
Combining logical and probabilistic AI is a necessity in statistical relational AI, neuro-symbolic AI and explainable AI \cite{de_raedt_probabilistic_2008,getoor_introduction_2019,sarker_neuro-symbolic_2022}.
Probability theory with exponential families typically focuses on distributions with maximal support \cite{koller_probabilistic_2009}. 
In contrast, logic can be understood as a study of the support of models, marking the worlds consistent with a knowledge base.
Logical and probabilistic models can be treated in the tensor network formalism, which thus serves as a unifying representation language \cite{goessmann_tensor_2026}.

We in this work derive tensor network representations of all discrete maximum entropy distributions, based on the framework presented in \cite{goessmann_tensor-network_2025,goessmann_tensor_2026}.
Tensor network formats originate from quantum many-body physics \cite{white_density-matrix_1993} as a mitigation scheme for the curse of dimensionality \cite{bellman_adaptive_1961}.
From a numerical perspective \cite{hackbusch_new_2009,hackbusch_tensor_2012} formats such as the $\cpformat$ format \cite{hitchcock_expression_1927,beylkin_algorithms_2005} and the $\ttformat$ format \cite{oseledets_tensor-train_2011} have been studied.
They have been applied to various areas such as in high-dimensional PDEs \cite{eigel_variational_2019,eigel_adaptive_2020}, or quantum simulation \cite{sander_quantum_2025}.
Lastly, tensor network schemes are applied to artificial intelligence \cite{badreddine_logic_2022,domingos_tensor_2025,ali_explicit_2025}.

\subsection{Tensor notation}

We introduce tensors based on index variables $\catvariableof{\catenumerator}$ to each of the $\catenumeratorin=\{0,\ldots,\catorder-1\}$ axes, each taking values in $[\catdimof{\catenumerator}]=\{0,\ldots,\catdimof{\catenumerator}-1\}$, where $\catdimof{\catenumerator}\in\nn$ is the dimension of the variable.
Tensors are then maps
\begin{align*}
    \hypercorewith \defcols \facstates \rightarrow \rr
\end{align*}
from indices $\shortcatindices=(\catindexof{0},\ldots,\catindexof{\catorder-1})\in\facstates$ to their coordinates $\hypercoreat{\indexedshortcatvariables}\in\rr$.
One-hot encodings of index tuples $\shortcatindices$ are specific tensors $\onehotmapofat{\shortcatindices}{\shortcatvariables}$ such that $\onehotmapofat{\shortcatindices}{\shortcatvariables=\tildecatindexof{[\catorder]}}=1$ if $\tildecatindexof{[\catorder]}=\shortcatindices$ and $0$ else.

Given a hypergraph $\graph=\big(\nodes,\edges\big)$ of an arbitrary node set $\nodes$ and hyperedges $\edgein$ as arbitrary subsets of $\nodes$, we define decorating tensor network as follows.
Each node $\nodein$ is decorated by a variable $\catvariableof{\node}$ of dimension $\catdimof{\node}$ and each hyperedge $\edgein$ by a tensor $\hypercoreofat{\edge}{\catvariableof{\edge}}$ with the index variables to the contained nodes.
The set of decorating tensors then builds the tensor network $\extnet$ on the hypergraph.
A contraction of the tensor network keeping the nodes $\arbset\subset\nodes$ open is then the tensor $\contractionof{\extnet}{\catvariableof{\arbset}}$ with variables $\catvariableof{\arbset}$ and coordinates to the indices $\catindexof{\arbset}$ by
\begin{align*}
    \contractionof{\extnet}{\indexedcatvariableof{\arbset}}
    = \sum_{\catindexof{\nodes\backslash\arbset}} \prod_{\edgein} \hypercoreofat{\edge}{\indexedcatvariableof{\edge}} \, .
\end{align*}
We in this work always assume that to each node in the hypergraph there is at least one hyperedge containing it.

\subsection{Maximum entropy problem}

Given a non-negative and non-vanishing base measure $\basemeasurewith$, a probability distribution is a non-negative tensor $\probwith$ such that
\begin{align*}
    \contraction{\probwith,\basemeasurewith} = 1 \, .
\end{align*}
We denote the set of such distributions by
\begin{align*}
    \distmeasuredby{\basemeasure}
    = \big\{ \probwith \wcols &\probwith \geq \zerosat{\shortcatvariables} \ncond \contraction{\probwith,\basemeasurewith} = 1 \ncond \\
    &\forall_{\shortcatindicesin}: \, (\basemeasureat{\indexedshortcatvariables}=0) \Rightarrow (\probat{\indexedshortcatvariables}=0)  \big\} \, .
\end{align*}
Note that we here restrict the support of the tensor $\probwith$ to be included in the support of the base measure, to avoid ambiguity in the representation.

The entropy of a distribution relative to the base measure $\basemeasure$ is
\begin{align*}
    \sentropyofwrt{\probwith}{\basemeasure}
    = -\contraction{\probwith,\lnof{\probwith},\basemeasurewith} \, .
\end{align*}
Note that this quantity depends on the chosen base measure.

A statistic is a function $\sstat:\facstates\rightarrow\rr^\seldim$ assigning a real vector of dimension $\seldim$ to each state $\shortcatindicesin$.
We denote by $\sstatcoordinateof{\selindex}$ the restriction to the $\selindex$-th coordinate of $\sstat$.
The mean parameter of a distribution $\probwith$ to a statistic $\sstat$ is the vector $\meanparamwith\in\rr^\seldim$ with the coordinates
\begin{align*}
    \meanparamat{\indexedselvariable}
    = \expectationof{\enumformula}
    = \contraction{\sstatcoordinateofat{\selindex}{\shortcatvariables},\probwith,\basemeasurewith} \, .
\end{align*}
To express the computation of the mean parameter as a single contraction, we introduce a coordinate selection variable $\selvariable$ of dimension $\seldim$ and the selection encoding $\sencsstatwith$ with coordinates by
\begin{align*}
    \sencsstatat{\indexedshortcatvariables,\indexedselvariable}
    = \sstatcoordinateofat{\selindex}{\indexedshortcatvariables} \, .
\end{align*}
The mean parameter is then
\begin{align*}
    \meanparamwith
    = \contractionof{\probwith,\sencsstatwith}{\selvariable} \, .
\end{align*}

The maximum entropy problem given a mean parameter $\genmeanwith$ is stated by
\begin{equation}
    \tag{$\mathrm{P}_{\sstat,\meanparam,\basemeasure}$}\label{prob:maxEntropy}
    \argmax_{\probwith\in\distmeasuredby{\basemeasure}} \sentropyofwrt{\probwith}{\basemeasure}
    \stspace
    \contractionof{\probwith,\sencsstatwith,\basemeasurewith}{\selvariable} = \genmeanat{\selvariable}
\end{equation}

\subsection{\ComputationActivationNetworks{}}

We now introduce \ComputationActivationNetworks{}, a family of tensor networks representing, as we will show in this work, the solution of any instance of \problemref{prob:maxEntropy}.
Given a statistic $\sstat:\facstates\rightarrow\rr^{\seldim}$ we define an index enumeration variable $\headvariableof{\selindex}$ of dimension $\cardof{\imageof{\sstatcoordinateof{\selindex}}}$ and injective index enumeration vectors
\begin{align*}
    \indexinterpretationofat{\selindex}{\headvariableof{\selindex}} \wcols [\headdimof{\selindex}] \rightarrow \rr \, .
\end{align*}
Based on the enumeration of the image we define the basis encoding tensor of the statistic as
\begin{align*}
    \sstatccwith
    = \sum_{\shortcatindicesin}
    \left(\bigotimes_{\selindexin}
        \onehotmapofat{\invindexinterpretationofat{\selindex}{\sstatcoordinateofat{\selindex}{\shortcatindices}}}{\headvariableof{\selindex}}\right)
    \otimes \onehotmapofat{\shortcatindices}{\shortcatvariables} \, .
\end{align*}
A computation network is any representation of $\sstatccwith$ as a tensor network.
These can be constructed in the case where the statistics are a composition of connective functions, see \cite{goessmann_tensor_2026}.

An activation tensor is $\hypercoreat{\headvariables}$ and the corresponding \ComputationActivationNetwork{} with respect to the statistic $\sstat$ and the base measure $\basemeasurewith$ the tensor
\begin{align*}
    \probwith
    = \frac{
        \contractionof{\sstatccwith,\hypercoreat{\headvariables}}{\shortcatvariables}
    }{
        \contractionof{\sstatccwith,\hypercoreat{\headvariables},\basemeasurewith}{\shortcatvariables}
    } \, .
\end{align*}

We are interested in decomposition formats of $\hypercoreat{\headvariables}$, where we use sets of tensor networks $\tnsetof{\graph}$ on a hypergraph $\graph$.

\begin{definition}
    The \ComputationActivationNetworks{} to a statistic $\sstat$ and a hypergraph $\graph$ is the family
    \begin{align*}
        \cansof{\sstat,\graph,\basemeasure}
        = \left\{
              \frac{\contractionof{\hypercoreat{\headvariables},\bencsstatwith}{\shortcatvariables}}{\contraction{\hypercoreat{\headvariables},\bencsstatwith,\basemeasurewith}}
              \wcols \hypercoreat{\headvariables} \in \tnsetof{\graph}
        \right\} \, .
    \end{align*}
\end{definition}

\subsection{Contributions}

We exploit the novel tensor network formalism developed in \cite{goessmann_tensor_2026} to represent generic maximum entropy distributions based on the position of their mean parameters in the mean polytope, see \theref{the:MAINgenMaxEntChar}.
The connection to logical formalisms of artificial intelligence is exploited in the study of Boolean statistics, where we interpret mean polytopes based on satisfiability of formulas (see \theref{the:boolVecSat}) and construct statistics to arbitrary polytopes (see \theref{the:boolStatToPolytope}).

\subsection{Structure of the paper}

To prepare for the presentation of our main results we introduce in section \secref{sec:meanPolytope} the mean polytope, its faces, and their tensor network representations.
In \secref{sec:expFamily} we study exponential families as classes of elementary \ComputationActivationNetworks{} and show that they are the maximum entropy distributions to mean parameters in the relative interior of the mean polytope.
This result is generalized in \secref{sec:genMaxEnt} where we characterize generic maximum entropy distributions based on \ComputationActivationNetworks{}.
In \secref{sec:booleanStat} we investigate consequences of these results in case of Boolean statistics and connect them to propositional logics.

%% file: sections/mean-polytope.tex
\section{The mean polytope}\label{sec:meanPolytope}

Given a statistic function
\begin{align*}
    \sstat \wcols \facstates \rightarrow \rr^{\seldim}
\end{align*}
on the states $\facstates$ and a base measure $\basemeasurewith$ the mean polytope is the set
\begin{align*}
    \genmeanset
    \coloneqq \left\{\contractionof{\probwith,\sencsstatwith,\basemeasurewith}{\selvariable}\wcols\probwith\in\distmeasuredby{\basemeasure} \right\}
\end{align*}
containing all mean parameters for the distributions in $\distmeasuredby{\basemeasure}$.
By $\sencsstatwith$ we denote the selection encoding of the statistic (see introduction).
Since $\distmeasuredby{\basemeasure}$ is the convex hull of the vertices $\frac{1}{\basemeasureat{\indexedshortcatvariables}} \cdot \onehotmapofat{\shortcatindices}{\shortcatvariables}$, and any mean parameter is a linear transformation of the distribution, the mean polytope is the convex hull of the vectors
\begin{align*}
    \genimset
    \coloneqq \{\sencsstatat{\indexedshortcatvariables,\selvariable}
    \wcols\shortcatindices\in\facstates\ncond\basemeasureat{\indexedshortcatvariables}\neq0 \} \, .
\end{align*}

%
%

\subsection{Faces}


When studying maximum entropy distributions, we will pay special attention to the faces of a polytope, which form the boundary.
Here, we define them based on convex hulls and introduce tensor representations by refined base measures.

\begin{definition}
    We say that the convex hull of a subset $\genfaceimset\subset\genimset$ is a face of a polytope $\genmeanset$, if and only if there is a vector $\canparamofat{\facesymbol}{\selvariable}\in\parspace$ such that
    \begin{align*}
        \genfaceimset
        = \argmax_{\meanparamwith\in\genimset} \contraction{\meanparamwith,\canparamofat{\facesymbol}{\selvariable}} \, .
    \end{align*}
    We denote the face as $\facesymbol=\convhullof{\genfaceimset}$.
    The set of all faces of the mean polytope $\genmeanset$ is denoted by $\facelatticeof{\genmeanset}$.
\end{definition}

The set $\facelatticeof{\genmeanset}$, partially ordered by the inclusion operation, is a lattice (see \cite{ziegler_lectures_2013}).
We now show that each face itself is a convex polytope corresponding to a refined base measure.

\begin{definition}
    \label{def:faceMeasure}
    The refined base measure with respect to a face $\facein$ is the tensor $\genfacemeasurewith$ with coordinates to $\shortcatindicesin$ by
    \begin{align*}
        \genfacemeasurewith
        = \begin{cases}
              \basemeasureat{\indexedshortcatvariables} & \ifspace \sstatat{\shortcatindices}\in\genfaceimset \\
              0 & \elsetext
        \end{cases} \, .
    \end{align*}
\end{definition}

The refined base measure has by definition a smaller support than the base measure.
We now show that restricting the distributions in $\distmeasuredby{\basemeasure}$ to those in $\distmeasuredby{\basemeasureof{\sstat,\facesymbol}}$ amounts to a restriction of the polytope $\genmeanset$ to the corresponding face.


\begin{lemma}
    \label{lem:faceAsRefinedPolytope}
    For any face $\facein$ we have
    \begin{align*}
        \facesymbol
        = \meansetof{\sstat,\basemeasureof{\sstat,\facesymbol}} \, .
    \end{align*}
\end{lemma}
\begin{proof}
    For any $\shortcatindicesin$ we have $\genfacemeasureat{\indexedshortcatvariables}\neq0$ if and only if $\sstatat{\shortcatindices}\in\genfaceimset$ and $\basemeasureat{\indexedshortcatvariables}\neq0$.
    Thus we have
    \begin{align*}
        \facesymbol
        &= \convhullof{\genfaceimset} \\
        &= \convhullof{\sencsstatat{\indexedshortcatvariables,\selvariable} \wcols \sstatat{\shortcatindices}\in\genfaceimset\ncond \basemeasureat{\indexedshortcatvariables}\neq 0} \\
        &= \convhullof{\sencsstatat{\indexedshortcatvariables,\selvariable} \wcols \genfacemeasureat{\indexedshortcatvariables}\neq 0} \\
        &= \meansetof{\sstat,\basemeasureof{\sstat,\facesymbol}} \, . \qedhere
    \end{align*}
\end{proof}

Representability of a distribution with respect to face measures is an equivalent condition for the mean parameter of a distribution to be on a face, as we show next.

\begin{lemma}
    \label{lem:faceMeasureRepCondition} 
    If and only if for a distribution $\probwith\in\distmeasuredby{\basemeasure}$ and a face $\facein$ we have
    \begin{align*}
        \contractionof{\sencsstatwith,\probwith,\basemeasurewith}{\selvariable}\in\facesymbol\, ,
    \end{align*}
    then 
    \begin{align*}
        \probwith \in \distmeasuredby{\genfacemeasure} \, .
    \end{align*}
\end{lemma}
\begin{proof}
    For a distribution $\probwith\in\distmeasuredby{\basemeasure}$ to be in $\distmeasuredby{\genfacemeasure}$, it is enough to show that $\probwith$ is only supported at $\genfacemeasurewith$.
    We have
    \begin{align*}
        \meanparamat{\selvariable}
        = \sum_{\shortcatindices} \probat{\indexedshortcatvariables}\cdot\basemeasureat{\indexedshortcatvariables} \cdot \genstatshortcatencoding \, .
    \end{align*}
    Now, let $\canparamofat{\facesymbol}{\selvariable}$ be a face normal to the face $\facesymbol$.
    We then have
    \begin{align*}
        \contraction{\meanparamwith,\canparamofat{\facesymbol}{\selvariable}}
        = \sum_{\shortcatindicesin}
        \probat{\indexedshortcatvariables}\cdot\basemeasureat{\indexedshortcatvariables} \cdot \contraction{\genstatshortcatencoding,\canparamofat{\facesymbol}{\selvariable}} \, .
    \end{align*}

    Now if and only if $\probat{\indexedshortcatvariables}\cdot\basemeasureat{\indexedshortcatvariables}$ is supported only for $\shortcatindices$ with $\sstatat{\shortcatindices}\in\genfaceimset$ we have that
    \begin{align*}
        \contraction{\meanparamwith,\canparamofat{\facesymbol}{\selvariable}}
        = \max_{\meanparamwith\in\genmeanset} \contraction{\meanparamwith,\canparamofat{\facesymbol}{\selvariable}}
    \end{align*}
    which is equal to $\meanparamwith\in\facesymbol$.
    Thus, if and only if $\meanparamwith\in\facesymbol$ then $\probat{\shortcatvariables}$ is only supported at $\shortcatindices$ in the support of $\genfacemeasurewith$.
\end{proof}

Let us now investigate tensor network representations of refined base measures, based on the basis encoding $\bencodingof{\sstat}$ of a statistic.

\begin{theorem}
    \label{the:faceMeasureCharacterization}
    For any face $\facesymbol$ of $\meanset$ we have
    \begin{align*}
        \genfacemeasureat{\shortcatvariables}
        =\contractionof{\sstatccwith,\kcoreofat{\facesymbol}{\headvariables},\basemeasurewith}{\shortcatvariables}
    \end{align*}
    where
    \begin{align*}
        \kcoreofat{\facesymbol}{\headvariables}
        = \sum_{\imelementwith\in\genfaceimset} \onehotmapofat{\imelementwith}{\headvariables} \, .
    \end{align*}
\end{theorem}
\begin{proof}
    For any $\imelementwith\in\genfaceimset$ the tensor
    \begin{align*}
        \hypercoreofat{\imelement}{\shortcatvariables}
        = \contractionof{\sstatccwith,\onehotmapofat{\imelementwith}{\headvariables}}{\shortcatvariables}
    \end{align*}
    is the indicator of the preimage of $\meanparam$ under $\sstatencoding$.
    Since preimages of the elements in $\genfaceimset$ are disjoint, the support of $\hypercoreofat{\imelement}{\shortcatvariables}$ is disjoint and their sum
    \begin{align*}
        \sum_{\imelementwith\in\genfaceimset} \hypercoreofat{\imelement}{\shortcatvariables}
    \end{align*}
    is the indicator of the preimage of $\facesymbol$ under $\sstatencoding$.
    The face measure obeys thus
    \begin{align*}
        \genfacemeasurewith
        &= \contractionof{\left(
                              \sum_{\imelementwith\in\genfaceimset} \hypercoreofat{\imelement}{\shortcatvariables}
        \right), \basemeasurewith}{\shortcatvariables} \\
        &= \sum_{\imelementwith\in\genfaceimset}
        \contractionof{\sstatccwith,\onehotmapofat{\imelementwith}{\headvariables},\basemeasurewith}{\shortcatvariables} \\
        & = \contractionof{\sstatccwith,\kcoreofat{\facesymbol}{\headvariables},\basemeasurewith}{\shortcatvariables} \qedhere
    \end{align*}
\end{proof}

We now investigate the representation of refined base measures by \ComputationActivationNetworks{}.

\begin{definition}
    \label{def:faceRepresentability}
    Let $\graph$ be a hypergraph whose nodes include $[\seldim]$.
    We say that a face $\facesymbol$ is representable by $\graph$ if and only if there is a set $\arbset$ of basis vectors with $\arbset\cap\genimset = \genfaceimset$ and there is a tensor network $\kcoreof{\graph}$ with respect to the hypergraph $\graph$ such that
    \begin{align*}
        \contractionof{\kcoreof{\graph}}{\headvariables}
        = \sum_{\imelement\in\arbset} \onehotmapofat{\imelement}{\headvariables} \, . 
    \end{align*}
    We call any such tensor network a face activating tensor network.
\end{definition}

\begin{lemma}
    \label{lem:faceMeasureRepresentation}
    If and only if a face is representable by a hypergraph $\graph$, we have
    \begin{align*}
        \frac{1}{\contraction{\genfacemeasurewith}} \cdot \genfacemeasurewith
        \in \cansof{\sstat,\graph,\basemeasure} \, .
    \end{align*}
\end{lemma}
\begin{proof}
    For any $\arbset$ with $\arbset \wcols \arbset\cap\genimset = \genfaceimset$ and tensor network $\kcoreof{\graph}$ respecting the assumptions of \defref{def:faceRepresentability} we have
    \begin{align*}
        \contractionof{\contractionof{\kcoreof{\graph}}{\headvariables},\bencsstatwith}{\shortcatvariables}
        &= \contractionof{\sum_{\imelement\in\arbset}\onehotmapofat{\imelement}{\headvariables},\bencsstatwith}{\shortcatvariables} \\
        &= \contractionof{\sum_{\imelement\in\genfaceimset}\onehotmapofat{\imelement}{\headvariables},\bencsstatwith}{\shortcatvariables} \, .
    \end{align*}
    Here we used in the second equation that only the vertices in $\genimset $ are in the image of $\sstat$.
    It follows, that
    \begin{align*}
        \frac{1}{\contraction{\genfacemeasurewith}} \cdot \genfacemeasurewith
        = \frac{
            \contractionof{\contractionof{\kcoreof{\graph}}{\headvariables},\bencsstatwith}{\shortcatvariables}
        }{
            \contraction{\contractionof{\kcoreof{\graph}}{\headvariables},\bencsstatwith,\basemeasurewith}
        }
    \end{align*}
    and thus $\genfacemeasureat{\shortcatvariables|\varnothing} \in \cansof{\sstat,\graph,\basemeasure}$.
\end{proof}

\input{examples/mean-polytope/vertices}

While vertices are the minimal non-vanishing faces in the face-lattice (see \cite{ziegler_lectures_2013}), we now show that the maximal face, namely the polytope itself, is also representable with respect to the elementary hypergraph $\elgraph$ (see \figref{fig:elFormat}).

\input{examples/mean-polytope/maximal-face}

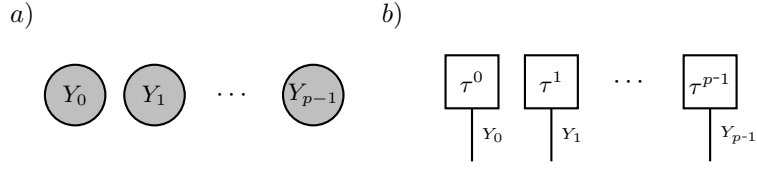
\begin{figure}
    \input{tikz/mean-polytope/el-format}
    \caption{Elementary tensor format consisting of the elementary hypergraph $\elformat$ $a)$, and corresponding tensor networks $b)$ being the tensor products of vectors.}
    \label{fig:elFormat}
\end{figure}

\subsection{Representation in the $\cpformat$ format}

Extending \exaref{exa:vertexMeasures} and \exaref{exa:maximalFaceMeasure}, we can provide a coarse estimation of the hypergraph $\graph$ required to decompose $\kcoreof{\facesymbol}$ for generic faces $\facesymbol$.
To this end, we use the $\cpformat$ format describing tensor networks on the $\cpformat$ hypergraph (see \figref{fig:cpFormat})
\begin{align*}
    \cpformat
    = \big([\seldim]\cup\{\decvariable\},\{\{\selindex,\decvariable\}\wcols\selindexin\}\big) \, .
\end{align*}
The dimension of the variable $\decvariable$ is called the hidden rank of the format.

\begin{lemma}
    \label{lem:CPfaceRepresentationBound} 
    Any face $\facesymbol$ is representable by a $\cpformat$ graph with hidden rank
    \begin{align*}
        r
        \leq \min\left(\cardof{\genfaceimset},\cardof{\genimset}-\cardof{\genfaceimset}+1\right) \, .
    \end{align*}
\end{lemma}
\begin{proof}
    We show the claim by constructing two face activating tensor networks to $\facesymbol$ in a $\cpformat$ hypergraph with hidden rank $\cardof{\genfaceimset}$ and in a $\cpformat$ hypergraph with hidden rank $\cardof{\genimset}-\cardof{\genfaceimset}+1$.
    To show the first representation we enumerate the vertices by a variable $\decvariable$ with dimension $r=\cardof{\genfaceimset}$, i.e. $\genfaceimset = \{v^{\decindex}[\selvariable] \wcols \decindexin\}$.
    Then we define for $\selindexin$ core tensors $\hypercoreofat{\selindex}{\headvariableof{\selindex},\decvariable}$
    \begin{align*}
        \hypercoreofat{\selindex}{\headvariableof{\selindex},\indexeddecvariable} = \onehotmapofat{\imelementofat{\decindex}{\indexedselvariable}}{\headvariableof{\selindex}} \, .
    \end{align*}
    Then we have
    \begin{align*}
        \contractionof{\{\hypercoreofat{\selindex}{\headvariableof{\selindex},\decvariable}\wcols\selindexin\}}{\headvariables}
        =\sum_{\imelement\in\genfaceimset} \onehotmapofat{\imelement}{\headvariables}
    \end{align*}
    and thus have found an activation tensor network in a $\cpformat$ graph with hidden rank $\cardof{\genfaceimset}$ representing the face $\facesymbol$.

    We continue with the second representation, for which we enumerate the set $\genimset/\genfaceimset$ by $\imelementofat{\decindex}{\selvariable}$ where $\decindex\in[\cardof{\genimset}-\cardof{\genfaceimset}]$.
    We define variable $\decvariable$ with dimension $r=\cardof{\genimset}-\cardof{\genfaceimset}+1$ and define for $\selindexin$ core tensors
    \begin{align*}
        \hypercoreofat{\selindex}{\headvariableof{\selindex},\indexeddecvariable}
        =
        \begin{cases}
            -\onehotmapofat{\imelementofat{\decindex}{\indexedselvariable}}{\headvariableof{\selindex}} & \ifspace \decindex < \cardof{\genimset}-\cardof{\genfaceimset} \\
            \onesat{\headvariableof{\selindex}} & \ifspace \decindex = \cardof{\genimset}-\cardof{\genfaceimset}
        \end{cases} \, .
    \end{align*}
    We then have
    \begin{align*}
        \contractionof{\{\hypercoreofat{\selindex}{\headvariableof{\selindex},\decvariable}\wcols\selindexin\}}{\headvariables}
        &=\onesat{\headvariables} - \sum_{\imelement\in\genimset/\genfaceimset} \onehotmapofat{\imelement}{\headvariables} \\
        &= \sum_{\imelement\in\genfaceimset} \onehotmapofat{\imelement}{\headvariables}
        + \sum_{\imelement\in\left(\bigtimes_{\selindexin}[\seldimof{\selindex}]\right)/\genimset} \onehotmapofat{\imelement}{\headvariables} \, .
    \end{align*}
    We have thus found a face activating tensor network for $\facesymbol$ in a $\cpformat$ format with hidden rank $r=\cardof{\genimset}-\cardof{\genfaceimset}+1$.
\end{proof}

We notice that the $\cpformat$ rank bound of \lemref{lem:CPfaceRepresentationBound} is tight in \exaref{exa:vertexMeasures} and \exaref{exa:maximalFaceMeasure}, since the elementary format is a restriction of the $\cpformat$ format to rank $1$.

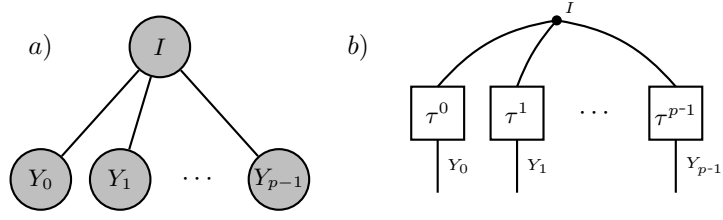
\begin{figure}
    \input{tikz/mean-polytope/cp-format}
    \caption{The $\cpformat$ tensor format consist of the $\cpformat$ hypergraph a) and decorating tensor networks b).
    Tensors in this format constructed as in \lemref{lem:CPfaceRepresentationBound} represent arbitrary faces.}
    \label{fig:cpFormat}
\end{figure}

\subsection{Examples of elementarily representable faces}

We show in the following example that all faces of a hypercube and a simplex are representable by elementary activation tensors.

\input{examples/mean-polytope/hypercube-faces}

\input{examples/mean-polytope/simplex-faces}

\subsection{Partition into relative interiors of faces}

Let us now introduce relative interiors, which enable us to find disjoint partitions of the mean polytope.

\begin{definition}[Relative Interior]
    \label{def:relativeInterior}
    The relative interior $\sbinteriorof{\arbset}$ of an arbitrary set $\arbset\subset\parspace$ is the interior of $\arbset$ in the affine hull of $\arbset$. 
\end{definition}

\begin{lemma}
    \label{lem:faceToMean}
    To each $\meanparam\in\genmeanset$ there is a unique face $\facesymbol$, which we denote by $\faceto{\meanparam}$, such that $\meanparam\in\sbinteriorof{\facesymbol}$.
\end{lemma}
\begin{proof}
    For any mean vector $\meanparam$ on a face $\facesymbol$ we either have $\meanparam\in\sbinteriorof{\facesymbol}$ or there is a face $\secfacesymbol$ of lower affine dimension with $\secfacesymbol\subset\facesymbol$ and $\meanparam\in\secfacesymbol$.
    Starting from the face $\genmeanset$ (for which by assumption $\meanparam\in\genmeanset$) we iteratively choose a face $\facesymbol$ of lower dimension containing $\meanparam$ until we reach a face with $\meanparam\in\sbinteriorof{\facesymbol}$.
    This procedure is guaranteed to terminate, since the affine dimension is decreasing and any face of affine dimension $0$ coincides with its relative interior.
\end{proof}

%
%

%% file: examples/mean-polytope/vertices.tex
\begin{example}[Vertices]
    \label{exa:vertexMeasures}
    Vertices $\facesymbol=\{\imelementwith\}$ are proper faces of affine dimension $0$, that is they consist of single vectors $\imelementwith$.
    Since each vertex is in the image $\genimset$, there exists an index tuple $\shortcatindices\in\stateset$ such that $\basemeasureat{\indexedshortcatvariables}\neq1$ and
    \begin{align*}
        \imelementwith
        = \sencsstatat{\indexedshortcatvariables,\selvariable} \, .
    \end{align*}
    Then $\kcoreofat{\facesymbol}{\headvariables}$ is the one-hot encoding of the vector $\sencsstatat{\indexedshortcatvariables,\selvariable}$, that is
    \begin{align*}
        \kcoreofat{\facesymbol}{\headvariables}
        = \onehotmapofat{\sencsstatat{\indexedshortcatvariables,\selvariable}}{\headvariables}
        = \contractionof{\{\onehotmapofat{\sencsstatat{\indexedshortcatvariables,\indexedselvariable}}{\headvariableof{\selindex}}\wcols\selindexin\}}{\headvariables}\, .
    \end{align*}
    Such tensors are examples of elementary tensors corresponding to the elementary hypergraph (see \figref{fig:elFormat})
    \begin{align*}
        \elformat
        = \big([\seldim],\{\{\selindex\}\wcols\selindexin\}\big) \, .
    \end{align*}
    In particular, the activation core is elementary and the corresponding refined base measure is in $\realizabledistsof{\sstat,\elgraph,\basemeasure}$.
\end{example}

%% file: examples/mean-polytope/maximal-face.tex
\begin{example}[Maximal face]
    \label{exa:maximalFaceMeasure}
    The maximal face $\facesymbolof{\varnothing}=\genmeanset$ coincides with the mean polytope itself and is given by the choice $\canparamofat{\varnothing}{\selvariable}=\zerosat{\selvariable}$.
    In this case the corresponding activation tensor to the face measure is
    \begin{align*}
        \kcoreofat{\varnothing}{\headvariables}
        = \onesat{\headvariables}
        = \bigotimes_{\selindexin}\onesat{\headvariableof{\selindex}} \, ,
    \end{align*}
    where by $\onesat{\headvariables}$ we denote tensors with all coordinates $1$.
    $\kcoreof{\varnothing}{\headvariables}$ is elementary and the maximal face is always representable by an elementary activation.
\end{example}

%% file: tikz/mean-polytope/el-format.tex
\stantikzpic{

    \drawtextnode{-2}{3}{$a)$}

    \drawvariablenode{0}{0}{$\headvariableof{0}$}{X0}
    \drawvariablenode{3}{0}{$\headvariableof{1}$}{X1}
    \drawtextnode{6}{0}{$\cdots$}
    \drawvariablenode{9}{0}{$\headvariableof{\seldim-1}$}{Xdmin1}

    \begin{scope}[shift={(15,-2)}]
        \drawtextnode{-3}{5}{$b)$}

        \drawtensorblock{0}{2.5}{1}{$\hypercoreof{0}$}
        \drawvwire{0}{1.5}{$\headvariableof{0}$}
        \drawtensorblock{3}{2.5}{1}{$\hypercoreof{1}$}
        \drawvwire{3}{1.5}{$\headvariableof{1}$}

        \drawtextnode{6}{2.5}{$\cdots$}

        \drawtensorblock{9}{2.5}{1}{$\hypercoreof{\seldim\shortminus1}$}
        \drawvwire{9}{1.5}{$\headvariableof{\seldim\shortminus1}$}

    \end{scope}

}

%% file: tikz/mean-polytope/cp-format.tex
\stantikzpic{

    \drawtextnode{0}{5}{$a)$}

    \drawvariablenode{0}{0}{$\headvariableof{0}$}{X0}
    \drawvariablenode{3}{0}{$\headvariableof{1}$}{X1}
    \drawtextnode{6}{0}{$\cdots$}
    \drawvariablenode{9}{0}{$\headvariableof{\seldim-1}$}{Xdmin1}

    \drawvariablenode{4.5}{5}{$\decvariable$}{I}

    \draw (X0) -- (I);
    \draw (X1) -- (I);
    \draw (Xdmin1) -- (I);

    \begin{scope}[shift={(15,0)}]
        \drawtextnode{-3}{5}{$b)$}

        \drawtensorblock{0}{2.5}{1}{$\hypercoreof{0}$}
        \drawvwire{0}{1.5}{$\headvariableof{0}$}
        \drawtensorblock{3}{2.5}{1}{$\hypercoreof{1}$}
        \drawvwire{3}{1.5}{$\headvariableof{1}$}

        \drawtextnode{6}{2.5}{$\cdots$}

        \drawtensorblock{9}{2.5}{1}{$\hypercoreof{\seldim\shortminus1}$}
        \drawvwire{9}{1.5}{$\headvariableof{\seldim\shortminus1}$}

        \drawvariabledot{4.5}{6}
        \drawtextnode{5}{6.5}{\colorlabelsize $\decvariable$}
        \draw (0,3.5) to[bend right=-20] (4.5,6);
        \draw (3,3.5) to[bend right=-10] (4.5,6);
        \draw (9,3.5) to[bend right=20] (4.5,6);
    \end{scope}

}

%% file: examples/mean-polytope/hypercube-faces.tex
\begin{example}[Hypercube]
    \label{exa:hypercubeFaces}
    In cases where $\shortcatvariables$ are Boolean and we have $\selindexin$ features
    \begin{align*}
        \formulaofat{\selindex}{\indexedshortcatvariables} = \catindexof{\selindex}
    \end{align*}
    the mean polytope is the hypercube
    \begin{align*}
        \meansetof{\{\formulaof{\selindex}\wcols\selindexin\},\trivbm} = \fullparcube \, .
    \end{align*}
    The non-empty faces in the face lattice $\facelatticeof{\fullparcube}$ can be enumerated by subsets $\arbset\subset[\seldim]$ and indices $\headindexof{\arbset}\in\bigtimes_{\selindex\in\arbset}[2]$ and represented by the cartesian products
    \begin{align*}
        \facesymbolof{(\arbset,\headindexof{\arbset})}
        = \bigtimes_{\selindexin} \mathcal{I}^{l,(\arbset,\headindexof{\arbset})}
    \end{align*}
    where
    \begin{align*}
        \mathcal{I}^{l,(\arbset,\headindexof{\arbset})}
        = \begin{cases}
        [0,1]
              & \ifspace \selindex\notin\arbset \\
              \{\headindexof{\selindex}\}& \ifspace \selindex\in\arbset
        \end{cases} \, .
    \end{align*}
    Each of these faces can be represented with respect to the elementary graph $\elgraph$, namely by the tensor product of leg vectors
    \begin{align*}
        \hardactlegwith
        = \begin{cases}
              \onesat{\headvariableof{\selindex}} & \ifspace \selindex\notin\arbset \\
              \onehotmapofat{\headindexof{\selindex}}{\headvariableof{\selindex}}& \ifspace \selindex\in\arbset
        \end{cases} \, .
    \end{align*}
\end{example}

%% file: examples/mean-polytope/simplex-faces.tex
\begin{example}[Simplices]
    \label{exa:simplexFaces}
    The $(\seldim-1)$-dimensional standard simplex $\meansetof{\triangle,\seldim-1}$ in $[0,1]^{\seldim}$ is the convex hull of the vertex sets
    \begin{align*}
        \imset = \{\onehotmapofat{\selindex}{\selvariable} \wcols\selindexin\} \, .
    \end{align*}
    The simplex is the mean polytope of the $\seldim$-dimensional statistic with
    \begin{align*}
        \sstatcoordinateof{\selindex} \defcols [\seldim] \rightarrow [2] \quad, \quad
        \sstatcoordinateofat{\selindex}{\secselindex}
        = \begin{cases}
              1 & \ifspace \secselindex = \selindex \\
              0 & \ifspace \secselindex \neq \selindex
        \end{cases} \, .
    \end{align*}
    The simplex represents all distributions with respect to the trivial base measure $\trivbm$ which coordinates are all $1$.
    When flattened to a vector, any tensor in $\distmeasuredby{\ones}$ is an element of $\meansetof{\triangle,\seldim-1}$ where $\seldim=\prod_{\catenumeratorin}\catdimof{\catenumerator}$.

    The face lattice of the simplex is enumerated by subsets of $[\seldim]$ as
    \begin{align*}
        \facelatticeof{\meansetof{\triangle,\seldim-1}}
        = \{\facesymbolof{\arbset}\wcols\arbset\subset[\seldim]\}
    \end{align*}
    where the faces are
    \begin{align*}
        \facesymbolof{\arbset}
        = \convhullof{\onehotmapofat{\selindex}{\selvariable} \wcols\selindex\in\arbset}
    \end{align*}
    and are therefore themselves $(\cardof{\arbset}-1)$-dimensional standard simplices.
    The partial order of the faces coincides with the inclusion order of subsets $\arbset$.
\end{example}

%% file: sections/exp-family.tex
\section{Exponential families}\label{sec:expFamily}

It is well known that exponential families are specific maximum entropy distributions.
However, not every maximum entropy distribution is an exponential family member.
To prepare for the characterization of generic maximum entropy distributions, we will in this section represent exponential family members as \ComputationActivationNetwork{}s.
We further show that a maximum entropy distribution is an exponential family member if and only if its mean parameter is in the relative interior of the mean polytope.

\subsection{Representation as elementary \ComputationActivationNetworks{}}

Let us first define exponential families as classes of probability distributions and then derive their representation by elementary \ComputationActivationNetworks{}.

\begin{definition}[Exponential Family]
    \label{def:expFamily}
    Given a statistic function
    \begin{align*}
        \sstat \defcols \facstates \rightarrow \parspace
    \end{align*}
    and a base measure $\basemeasurewith$ with $\contraction{\basemeasure}\neq0$, the set $\expfamily=\{\expdist\wcols\canparamwithin\}\subset\distmeasuredby{\basemeasure}$ of probability distributions
    \begin{align*}
        \expdistat{\shortcatvariables}
        = \frac{
            \expof{\contractionof{\sencsstatat{\shortcatvariables,\selvariable},\canparamwith}{\shortcatvariables}}
        }{
            \contraction{\expof{\contractionof{\sencsstatat{\shortcatvariables,\selvariable},\canparamwith}{\shortcatvariables}},\basemeasurewith}
        }
    \end{align*}
    is called the exponential family with respect to $\sstat$.
\end{definition}

To connect with standard notation of exponential families we notice that the coordinate of any distribution $\expdistat{\shortcatvariables}$ at some state $\shortcatindicesin$ is
\begin{align*}
    \expdistat{\indexedshortcatvariables}
    = \expof{\sum_{\selindexin}\weightparamat{\indexedselvariable}\cdot\sstatcoordinateofat{\selindex}{\shortcatindices} - \psi_{\sstat,\basemeasure}(\weightparam)}
\end{align*}
where
\begin{align*}
    \psi_{\sstat,\basemeasure}(\weightparam)
    = \lnof{\contraction{\expof{\contractionof{\sencsstatat{\shortcatvariables,\selvariable},\canparamwith}{\shortcatvariables}},\basemeasurewith}}
\end{align*}
is the cumulant generating function.
The function $\psi_{\sstat,\basemeasure}(\weightparam)$ and its Legendre transform $\psi^*_{\sstat,\basemeasure}(\meanparam)$ define the dually flat Riemannian structure on the exponential family manifold \cite{amari_information_2016}.
We now continue with the notation of \defref{def:expFamily} to find representations of exponential families by elementary \ComputationActivationNetworks{}.

\begin{theorem}[Exponential Families are elementary \ComputationActivationNetworks{}]
    \label{the:expFamilyTensorRep}
    Given any base measure $\basemeasure$ and a sufficient statistic $\sstat$ we enumerate for each coordinate $\selindexin$ the image $\imageof{\sstatcoordinateof{\selindex}}$ by a variable $\headvariableof{\selindex}$ taking values in $[\cardof{\imageof{\sstatcoordinateof{\selindex}}}]$, given an interpretation map
    \begin{align*}
        \indexinterpretationof{\selindex} \defcols
        [\cardof{\imageof{\sstatcoordinateof{\selindex}}}] \rightarrow \imageof{\sstatcoordinateof{\selindex}} \, .
    \end{align*}
    For any canonical parameter vector $\canparamwithin$ we build the activation cores $\softactlegwith$ for each coordinate $\headindexof{\selindex}\in[\cardof{\imageof{\sstatcoordinateof{\selindex}}}]$ by
    \begin{align*}
        \softactleg\left[\indexedheadvariableof{\selindex}\right]
        = \expof{\canparamat{\indexedselvariable} \cdot \indexinterpretationofat{\selindex}{\headindexof{\selindex}} } \,
    \end{align*}
    and have (see \figref{fig:expdistUnaryRealizable})
    \begin{align*}
        \expdistat{\shortcatvariables}
        = \frac{
            \contractionof{\{\bencsstatwith\}\cup\{\softactlegwith \wcols \selindexin\}}{\shortcatvariables}
        }{
            \contraction{\{\bencsstatwith\}\cup\{\softactlegwith \wcols \selindexin\}\cup\{\basemeasurewith\}}
        } \, .
    \end{align*}
\end{theorem}
\begin{proof} 
    For each $\shortcatindicesin$ we have
    \begin{align*}
        &\contractionof{\{\bencsstatwith\}\cup\{\softactlegwith \wcols \selindexin\}}{\indexedshortcatvariables} \\
        &\quad = \prod_{\selindexin} \contraction{\bencodingofat{\sstatcoordinateof{\selindex}}{\headvariableof{\selindex},\indexedshortcatvariables},\softactlegwith} \\
        &\quad = \prod_{\selindexin} \expof{\canparamat{\indexedselvariable}\cdot\indexinterpretationofat{\selindex}{\sstatcoordinateofat{\selindex}{\shortcatindices}}} \\
        &\quad = \expof{\sum_{\selindexin}\canparamat{\indexedselvariable}\cdot\indexinterpretationofat{\selindex}{\sstatcoordinateofat{\selindex}{\shortcatindices}}} \\
        &\quad = \expof{\contractionof{\sencsstatat{\shortcatvariables,\selvariable},\canparamwith}{\indexedshortcatvariables}} \, .
    \end{align*}
    Therefore we have
    \begin{align*}
        \contractionof{\{\bencsstatwith\}\cup\{\softactlegwith \wcols \selindexin\}}{\shortcatvariables}
        = \expof{\contractionof{\sencsstatat{\shortcatvariables,\selvariable},\canparamwith}{\shortcatvariables}} \, .
    \end{align*}
    The claim follows, since this implies that also the contraction of both sides with $\basemeasurewith$ is equivalent.
\end{proof}

\begin{figure}[t]
    \begin{center}
        \input{tikz/exp-family/expdist_unary_realizable}
    \end{center}
    \caption{Representation of a member of an exponential family by a \ComputationActivationNetwork{} with elementary activation.
    Here $\partitionfunction\in\rr$ is a normalization constant.}\label{fig:expdistUnaryRealizable}
\end{figure}

We will use the following well known property (see e.g. \cite{brown_fundamentals_1987}), that the mean parameters of the members of an exponential family are exactly the relative interior of the mean polytope.

\begin{lemma}
    \label{lem:interiorRepExpFamily}
    The set of mean parameters of the members of an exponential family is the relative interior of the mean polytope.
\end{lemma}
\begin{proof}
    See The~3.3 in \cite{wainwright_graphical_2008}.
\end{proof}

Based on this property we define the forward and backward mappings to an exponential family.

\begin{definition}\label{def:forwardBackwardMapping}
    The forward map of an exponential family is the map
    \begin{align*}
        \forwardmap  \defcols \parspace \rightarrow \sbinteriorof{\genmeanset}
    \end{align*}
    defined as
    \begin{align*}
        \forwardmapof{\canparam}
        = \contractionof{\expdistofat{\sstat,\canparam,\basemeasure}{\shortcatvariables},\sencsstatwith,\basemeasurewith}{\selvariable}
    \end{align*}
    Any map $\backwardmap: \sbinteriorof{\genmeanset}\rightarrow \parspace$ with $\forwardmapof{\backwardmapof{\meanparam}}=\meanparam$ for all $\meanparam\in\sbinteriorof{\genmeanset}$ is called a backward map.
\end{definition}

\subsection{Maximum entropy distributions}

One universal property of exponential families amongst others (see \cite{murphy_probabilistic_2023}) is that they maximize the entropy among those with the same mean vector.

\begin{lemma}\label{lem:expFamilyMaxEnt}
    Any member of an exponential family $\expfamilyof{\sstat,\basemeasure}$ is a maximum entropy distribution with respect to the moment constraint on $\sstat$ and the base measure $\basemeasure$.
\end{lemma}
\begin{proof}
    Let $\expdistof{\sstat,\canparam,\basemeasure}$ be an exponential family member with canonical parameter $\canparam$ and let $\meanparamwith$ be the corresponding mean parameter.
    Let $\secprobat{\shortcatvariables}$ be any feasible distribution for the maximum entropy problem with respect to the moment constraint on $\sstat$ by $\meanparamwith$ and the base measure $\basemeasure$.
    We also have $\contractionof{\secprobat{\shortcatvariables},\sencsstatwith,\basemeasurewith}{\selvariable}=\meanparamat{\selvariable}$ and thus
    \begin{align*}
        \centropyofwrt{\secprobtensor}{\expdistof{\sstat,\canparam,\basemeasure}}{\basemeasure}
        &= -\contraction{\secprobat{\shortcatvariables},\lnof{\expdistofat{\sstat,\canparam,\basemeasure}{\shortcatvariables}},\basemeasurewith} \\
        &= -\contraction{\secprobat{\shortcatvariables},\sencsstatwith,\canparamwith,\basemeasurewith} + \cumfunctionof{\canparam} \\
        &= - \contraction{\canparamwith,\meanparamwith} + \cumfunctionof{\canparam} \\
        &= \sentropyofwrt{\expdistof{\sstat,\canparam,\basemeasure}}{\basemeasure} \, .
    \end{align*}
    With the Gibbs inequality we have if $\secprobtensor\neq\expdistof{\sstat,\canparam,\basemeasure}$
    \begin{align*}
        \sentropyofwrt{\expdistof{\sstat,\canparam,\basemeasure}}{\basemeasure}  - \sentropyofwrt{\secprobtensor}{\basemeasure}
        = \centropyofwrt{\secprobtensor}{\expdistof{\sstat,\canparam,\basemeasure}}{\basemeasure}  - \sentropyofwrt{\secprobtensor}{\basemeasure}  > 0 \,
    \end{align*}
    and thus $\sentropyofwrt{\secprobtensor}{\basemeasure}<\sentropyofwrt{\expdistof{\sstat,\canparam,\basemeasure}}{\basemeasure}$.
    Therefore, any feasible $\secprobtensor$ different from $\expdistof{\sstat,\canparam,\basemeasure}$ has less entropy than $\expdistof{\sstat,\canparam,\basemeasure}$ and $\expdistof{\sstat,\canparam,\basemeasure}$ is a maximum entropy distribution.
\end{proof}

Together with \lemref{lem:interiorRepExpFamily} stating the existence of an exponential family member for any mean parameter in the relative interior of the mean polytope we arrive at the following theorem.

\begin{theorem}
    \label{the:maxEntropyInterior}
    If and only if $\meanparamwith$ is in the relative interior of $\genmeanset$, then the unique solution of the maximum entropy problem is the distribution
    \begin{align*}
        \expdistofat{\sstat,\weightparam,\basemeasure}{\shortcatvariables}\in\expfamilyof{\sstat,\basemeasure}
    \end{align*}
    where for some backward map $\backwardmapwrt{\sstat,\basemeasure}$
    \begin{align*}
        \weightparam
        = \backwardmapwrtof{\sstat,\basemeasure}{\meanparam} \, .
    \end{align*}
\end{theorem}
\begin{proof}
    By \lemref{lem:interiorRepExpFamily} if
    \begin{align*}
        \meanparamwith \in \sbinteriorof{\genmeanset}  \, ,
    \end{align*}
    there is at least one weight parameter $\weightparamwith\in\parspace$ with
    \begin{align*}
        \contractionof{\expdistofat{\sstat,\weightparam,\basemeasure}{\shortcatvariables},\sencsstatwith,\basemeasurewith}{\selvariable}
        =\meanparamat{\selvariable}
    \end{align*}
    and $\backwardmapwrtof{\sstat,\basemeasure}{\meanparam}$ with respect to any backward map $\backwardmapwrt{\sstat,\basemeasure}$ is such one.
    With \lemref{lem:expFamilyMaxEnt} we conclude that $\expdistof{(\sstat,\backwardmapwrtof{\sstat,\basemeasure}{\meanparam},\basemeasure)}$ is the maximum entropy distribution with respect to the moment constraint on $\sstat$ by $\meanparamwith$ and the base measure $\basemeasure$.
\end{proof}

While the maximum entropy distribution is unique, the weight vector $\weightparam$ representing the maximum entropy distribution is not always unique.
This is reflected by multiple possible choices of backward mappings.
Only if the exponential family is minimal (see \cite{wainwright_graphical_2008}) then the weight vector is unique.
However, we did not restrict to such situations, since exponential families on refined base measures are not minimal.

%% file: tikz/exp-family/expdist_unary_realizable.tex
\begin{tikzpicture}[scale=0.4,thick,xscale=1] 

    \begin{scope}
        [shift={(-11,0)}]
        \draw (-2,-1) rectangle (6,-3);
        \node[anchor=center] (text) at (2,-2) {\corelabelsize $ \expdist$};
        \draw[-<-] (0,-3)--(0,-5) node[midway,left] {\colorlabelsize $\catvariableof{0}$};
        \draw[-<-] (1.5,-3)--(1.5,-5) node[midway,left] {\colorlabelsize $\catvariableof{1}$};
        \node[anchor=center] (text) at (3,-4) {$\cdots$};
        \draw[-<-] (4,-3)--(4,-5) node[midway,right] {\colorlabelsize $\catvariableof{\atomorder\shortminus1}$};

        \node[anchor=center] (text) at (8,-2) {$= \frac{1}{\partitionfunction} \cdot $};
    \end{scope}

    \draw (-1.25,1) rectangle (1.25,3);
    \node[anchor=center] (text) at (0,2) {\corelabelsize $\softactsymbolof{0,\canparam}$};

    \draw (2.75,1) rectangle (5.25,3);
    \node[anchor=center] (text) at (4,2) {\corelabelsize $\softactsymbolof{\seldim\shortminus 1,\canparam}$};

    \draw[->-] (0,-1)--(0,0);
    \node[left] (text) at (0,0) {\colorlabelsize $\headvariableof{0}$};
    \draw[] (0,0)--(0,1);
    \drawvariabledot{0}{0}
    \node[anchor=center] (text) at (2,0) {$\cdots$};

    \draw[->-] (4,-1)--(4,0);
    \node[right] (text) at (4,0) {\colorlabelsize $\headvariableof{\seccatorder\shortminus1}$};
    \draw[] (4,0)--(4,1);
    \drawvariabledot{4}{0}

    \draw (-1,-1) rectangle (5,-3);
    \node[anchor=center] (text) at (2,-2) {\corelabelsize $\bencodingof{\sstat}$};
    \draw[-<-] (0,-3)--(0,-5) node[midway,left] {\colorlabelsize $\catvariableof{0}$};
    \draw[-<-] (1.5,-3)--(1.5,-5) node[midway,left] {\colorlabelsize $\catvariableof{1}$};
    \node[anchor=center] (text) at (3,-4) {$\cdots$};
    \draw[-<-] (4,-3)--(4,-5) node[midway,right] {\colorlabelsize $\catvariableof{\atomorder\shortminus1}$};

%
%
%

\end{tikzpicture}

%% file: sections/gen-rep.tex
\section{Generic maximum entropy distributions}\label{sec:genMaxEnt}

In the previous section we have shown that the members of an exponential family are maximum entropy distributions.
However, while their mean parameters constitute the relative interior of the mean polytope, they never lie on faces of lower dimension.
We now exploit the tensor representation study of such faces (see \secref{sec:meanPolytope}) to characterize the maximum entropy distributions in the general case.

\subsection{Main result}

Our main result generalizes the maximum entropy characterization of \theref{the:maxEntropyInterior} to arbitrary mean parameters.

\begin{theorem}[Generic characterization of Maximum Entropy Solutions]
    \label{the:MAINgenMaxEntChar}
    Let $\sstat$ be a statistic and $\basemeasure$ a base measure.
    For any $\weightparamwith$ and $\facein$ the distribution
    \begin{align*}
        \expdistofat{\sstat,\weightparam,\genfacemeasure}{\shortcatvariables}
    \end{align*}
    is a maximum entropy distribution. \\
    For any $\meanparamwith$ the maximum entropy problem has a feasible distribution, if and only if $\meanparamwith\in\genmeanset$.
    If $\meanparamwith\in\genmeanset$ the unique maximum entropy distribution is $\expdistof{\sstat,\weightparam,\genfacemeasure}$ where $\facesymbol=\faceto{\meanparam}$ (see \lemref{lem:faceToMean}) and
    \begin{align*}
        \weightparam
        = \backwardmapwrtof{\sstat,\genfacemeasure}{\meanparam}
    \end{align*}
    where $\genfacemeasure$ is the refined base measure (see \defref{def:faceMeasure}) and $\backwardmapwrt{\sstat,\genfacemeasure}$ is any backward map to the exponential family $\expfamilyof{\sstat,\genfacemeasure}$ (see \defref{def:forwardBackwardMapping}). \\
    For any hypergraph $\graph$ and any face $\facein$ we have that all maximum entropy distributions to mean parameters on $\interiorof{\facesymbol}$ are in $\cansof{\sstat,\graph,\basemeasure}$ if $\facesymbol$ is representable with respect to $\graph$ (see \defref{def:faceRepresentability}).
\end{theorem}

To prepare for the proof of this theorem we first show in an auxiliary lemma that we can reduce the set of feasible distributions in \problemref{prob:maxEntropy}.

\begin{lemma}
    \label{lem:maxEntReduction}
    For any $\meanparamwith\in\genmeanset$ and a face $\facein$ with $\meanparamwith\in\sbinteriorof{\facesymbol}$ the solutions of $\mathrm{P}_{\sstat,\meanparam,\basemeasure}$ and $\mathrm{P}_{\sstat,\meanparam,\genfacemeasure}$ coincide.
\end{lemma}
\begin{proof}
    By \lemref{lem:faceMeasureRepCondition} all feasible distributions are representable by the base measure refined to the face $\facesymbol$.
    We have that any for $\mathrm{P}_{\sstat,\meanparam,\basemeasure}$ feasible distribution $\probwith$ satisfies
    \begin{align*}
        \contraction{\probwith,\genfacemeasure} = 1 \,
    \end{align*}
    and thus $\probwith\in\distmeasuredby{\genfacemeasure}$.
    Conversely, any $\probwith\in\distmeasuredby{\genfacemeasure}$ satisfies
    \begin{align*}
        \contraction{\probwith,\genfacemeasure} = \contraction{\probwith,\basemeasure} = 1
    \end{align*}
    and thus $\probwith\in\distmeasuredby{\basemeasure}$.
    Problem $\mathrm{P}_{\sstat,\meanparam,\basemeasure}$ is thus equal to
    \begin{align*}
        \argmax_{\probwith\in\distmeasuredby{\genfacemeasure}} \sentropyofwrt{\probwith}{\basemeasure}
        \stspace
        \contractionof{\probwith,\sencsstatwith,\basemeasurewith}{\selvariable}
        = \genmeanat{\selvariable}
    \end{align*}
    We further have that any $\probwith\in\distmeasuredby{\genfacemeasure}$
    \begin{align*}
        \sentropyofwrt{\probwith}{\basemeasure} = \sentropyofwrt{\probwith}{\genfacemeasure}
    \end{align*}
    and arrive together with the above equivalence at the claim.
\end{proof}

\begin{proof}[Proof of \theref{the:MAINgenMaxEntChar}]
    To show that for any $\weightparamwith$ and $\facein$ the distribution $\expdistof{\sstat,\weightparam,\genfacemeasure}$ is a maximum entropy distribution, we apply \lemref{lem:maxEntReduction} and \theref{the:maxEntropyInterior}. \\
    If $\meanparamwith\notin\genmeanset$ then \problemref{prob:maxEntropy} is not feasible by definition.
    In case $\meanparamwith\in\genmeanset$ we find $\facein$ with $\facesymbol=\faceto{\meanparam}$ and by \lemref{lem:maxEntReduction} the solution of maximum entropy problem \problemref{prob:maxEntropy} coincides with $\mathrm{P}_{\sstat,\meanparam,\genfacemeasure}$.
    Since by \lemref{lem:faceAsRefinedPolytope} the face is the polytope
    \begin{align*}
        \facesymbol
        =\meansetof{\sstat,\genfacemeasure}
    \end{align*}
    we find a weight parameter
    \begin{align*}
        \weightparam
        = \backwardmapwrtof{\sstat,\genfacemeasure}{\meanparam} \, .
    \end{align*}
    and by \theref{the:maxEntropyInterior} $\expdistof{\sstat,\weightparam,\genfacemeasure}$ is the corresponding maximum entropy distribution. \\
    To show the last claim, it is enough to show that for any face $\facesymbol$ representable by $\graph$ and for any $\weightparamin$ we have
    \begin{align*}
        \expdistof{\sstat,\weightparam,\genfacemeasure}
        \in \cansof{\sstat,\graph,\basemeasure} \, .
    \end{align*}
    To this end, we construct activation cores as follows.
    By \defref{def:faceRepresentability} and \lemref{lem:faceMeasureRepresentation} there is a tensor network $\kcoreof{\graph}$ such that
    \begin{align*}
        \genfacemeasurewith
        = \contractionof{\kcoreof{\graph}\cup\{\basemeasurewith\}}{\shortcatvariables} \, .
    \end{align*}
    For each $\selindexin$ we find an edge $\edge(\selindex)\in\edges$ with $\selindex\in\edge(\selindex)$.
    Now we define a tensor network $\tnetof{\graph}$ by the cores to $\edgein$ as
    \begin{align*}
        \hypercoreofat{\edge}{\catvariableof{\edge}}
        = \contractionof{
            \{\kcoreofat{\edge}{\catvariableof{\edge}}\} \cup \{\softactlegwith\wcols\edge=\edge(\selindex)\}
        }{\catvariableof{\edge}} \, .
    \end{align*}
    Then, we have
    \begin{align*}
        \contractionof{\bencsstatwith,\contractionof{\tnetof{\graph}}{\headvariables}}{\shortcatvariables}
        = \contractionof{\bencsstatwith,\softacttensorwith,\contractionof{\kcoreof{\graph}}{\headvariables}}{\shortcatvariables}
    \end{align*}
    and therefore
    \begin{align*}
        \expdistof{\sstat,\weightparam,\genfacemeasure}
        = \frac{
            \contractionof{\{\bencsstatwith\}\cup\tnetof{\graph}}{\shortcatvariables}
        }{
            \contraction{\{\bencsstatwith,\basemeasurewith\}\cup\tnetof{\graph}}
        } \, .
    \end{align*}
    We conclude $\expdistof{\sstat,\weightparam,\genfacemeasure}\in \cansof{\sstat,\graph,\basemeasure}$.
\end{proof}

\subsection{Family of maximum entropy distributions}

The family of maximum entropy distributions is the set
\begin{align*}
    \maxentfamilyof{\sstat}{\basemeasure} =
    \{\probwith \wcols \exists\meanparam\in\genmeanset : \probwith \quad\text{is a solution of \problemref{prob:maxEntropy}} \} \, .
\end{align*}
\theref{the:MAINgenMaxEntChar} suggests a parametrization of this family as tuples of corresponding faces $\facesymbol$ and weight vectors $\weightparamwith$.
As sketched in \figref{fig:maxEntropyActcore} the activation tensor of each maximum entropy distribution is decomposed into a possibly non-elementary Boolean tensor corresponding to the base measure refinement on a face $\facesymbol$, and an elementary tensor dependent on the weight vector $\weightparamin$.
We thus define the canonical parameter space
\begin{align*}
    \gencanparamset
    \coloneqq \facelatticeof{\genmeanset} \times \weightspace
\end{align*}
and have
\begin{align*}
    \maxentfamilyof{\sstat}{\basemeasure} =
    \left\{
        \frac{
            \contractionof{\softacttensorat{\headvariables},\kcoreofat{\facesymbol}{\headvariables},\bencsstatwith}{\shortcatvariables}
        }{
            \contraction{\softacttensorat{\headvariables},\kcoreofat{\facesymbol}{\headvariables},\bencsstatwith,\basemeasurewith}
        }
        \wcols \canparamwithin,\facesymbol\in\gencanparamset
    \right\} \, .
\end{align*}

Note that these families are disjoint, since each member of an exponential family has support by the support of the face measure.

\begin{figure}[t]
    \begin{center}
        \input{tikz/main-result/max_entropy_actcore}
    \end{center}
    \caption{
        Tensor network decomposition of maximum entropy distributions to the constraint $\meanparamat{\selvariable}=\contractionof{\probwith,\sencsstatwith}{\selvariable}$.
        Blue: Boolean tensor $\kcoreofat{\facesymbol}{\headvariables}$ dependent on the face $\faceto{\meanparam}$, which represents the refinement of the base measure to the face, and is possibly non-elementary.
        Red: Elementary activation tensor $\softacttensorwith=\bigotimes_{\selenumeratorin}\softactlegwith$ to the weight $\weightparamwith= \backwardmapwrtof{\sstat,\genfacemeasure}{\meanparam}$.
    }\label{fig:maxEntropyActcore}
\end{figure}

%% file: tikz/main-result/max_entropy_actcore.tex
\begin{tikzpicture}[scale=0.35,thick]

    \begin{scope}
    [shift={(-15,0)}]

        \drawtensorblock{2}{-2}{3}{$\probtensor$}
        \drawddwire{0}{-5}{$\catvariableof{0}$}
        \drawddwire{1.5}{-5}{$\catvariableof{1}$}
        \drawddwire{4}{-5}{$\catvariableof{\catorder\shortminus1}$}
        \drawtextnode{3.25}{-4.75}{\colorlabelsize $\cdots$};

        \drawtextnode{8}{-2}{$= \,\frac{1}{\partitionfunction} \cdot $}
    \end{scope}

    \draw[\concolor] (-1,1) rectangle (5,3);
    \drawtextnode{2}{2}{\textcolor{\concolor}{$\kcoreof{\facesymbol}$}}

    \draw[->-] (0,-1)--(0,0);
    \node[right] (text) at (0,0) {\colorlabelsize $\headvariableof{0}$};
    \draw[\concolor] (0,0)--(0,1);
    \drawvariabledot{0}{0}
    \drawtextnode{2}{-0.5}{\colorlabelsize $\cdots$}

    \draw[\probcolor] (0,0) -- (-2,0);
    \draw[\probcolor] (-2,1) rectangle (-4,-1);
    \node[anchor=center,\probcolor] (text) at (-3,0) {\corelabelsize $\softactsymbolof{0,\canparam}$};

    \draw[->-] (4,-1)--(4,0);
    \node[left] (text) at (4,0) {\colorlabelsize $\headvariableof{\seccatorder\shortminus1}$};
    \draw[\concolor] (4,0)--(4,1);
    \drawvariabledot{4}{0}

    \draw[\probcolor] (4,0) -- (6,0);
    \draw[\probcolor] (6,1) rectangle (8.5,-1);
    \node[anchor=center,\probcolor] (text) at (7.25,0) {\corelabelsize $\softactsymbolof{\seccatorder\shortminus1,\canparam}$};

    \drawtensorblock{2}{-2}{3}{$\bencodingof{\sstat}$}
    \drawtdwire{0}{-3}{$\catvariableof{0}$}
    \drawtdwire{1.5}{-3}{$\catvariableof{1}$}
    \drawtdwire{4}{-3}{$\catvariableof{\catorder\shortminus1}$}
    \drawtextnode{3.25}{-4.75}{\colorlabelsize $\cdots$};

\end{tikzpicture}

%% file: sections/boolean-stat.tex
\section{Boolean statistics}\label{sec:booleanStat}

We say that a statistic is Boolean, if $\imageof{\sstatcoordinateof{\selindex}}\subset\ozset$ for all $\selindexin$.
The mean polytope is then a 0/1-polytope \cite{ziegler_lectures_2000}, which vertices are exactly the contained Boolean vectors.
The special case of Boolean statistics is especially interesting for two reasons.
First, we can provide additional intuition on those faces representable by elementary face activating tensors, namely based on intersections with cube faces.
Based on these properties we also show that any elementary \ComputationActivationNetwork{} to a Boolean statistic is a maximum entropy distribution.
Second, the statistic itself can be interpreted in terms of logical formulas.
We exploit this property to provide examples and to characterize each 0/1-polytope as a mean polytope to a collection of propositional formulas.

\subsection{Characterization of elementary representable faces}

Orienting on \exaref{exa:hypercubeFaces} we now introduce cube-likeness of faces and polytopes.

\begin{definition}
    \label{def:cubeLike}
    We say that a face $\facesymbol$ of $\genmeanset$ is cube-like, if it is empty or there is $\arbset\subset[\seldim]$ and $\headindexof{\arbset}\in\bigtimes_{\selindex\in\arbset}[2]$ such that
    \begin{align*}
        \facesymbol
        = \genmeanset \cap \facesymbolof{(\arbset,\headindexof{\arbset})} \, .
    \end{align*}
    Here we denote by $\facesymbolof{(\arbset,\headindexof{\arbset})}$ the faces of the hypercube $\cubeof{\seldim}$ (see \exaref{exa:hypercubeFaces}).
    We further say that a polytope $\genmeanset$ is cube-like, if all faces $\facesymbol\in\facelatticeof{\genmeanset}$ are cube-like.
\end{definition}

We now show that a face is cube-like if and only if it is representable by an elementary tensor.

\begin{theorem}
    \label{the:faceMeasureHardLogicNetworks}
    Let $\hlnstat$ be a Boolean statistic, $\basemeasure$ a base measure and $\facesymbol$ be a face of $\genmeanset$.
    Then the following are equivalent:
    \begin{itemize}
        \item[(i)] $\facesymbol$ is cube-like (see \defref{def:cubeLike}).
        \item[(ii)] $\facesymbol$ is representable by an elementary tensor (see \defref{def:faceRepresentability}).
    \end{itemize}
\end{theorem}
\begin{proof}
    If the face is empty, i.e. $\facesymbol=\varnothing$, it is by definition cube-like and has $\zerosat{\headvariables}$ as an elementary activation tensor.
    We therefore assume in the following $\facesymbol\neq\varnothing$.

    (i)$\Rightarrow$(ii):
    Let us assume that $\facesymbol$ is cube-like, that is there is $\arbset\subset[\seldim]$ and $\headindexof{\arbset}\in\bigtimes_{\selindex\in\arbset}[2]$ such that $\facesymbol = \genmeanset \cap \facesymbolof{(\arbset,\headindexof{\arbset})}$.
    We use that $\facesymbol,\genmeanset$ and $\facesymbolof{(\arbset,\headindexof{\arbset})}$ are the convex hulls of the cube vertex sets $\imsetof{\hlnstat,\meanset}{\facesymbol}$, $\imsetof{\hlnstat,\meanset}{}$ and that for the set
    \begin{align}\label{eq:cubeFaceVertices}
        \imsetof{}{\arbset,\headindexof{\arbset}}
        = \{\imelementwith \wcols \uniquantwrtof{\selindex\in\arbset}{\imelementat{\indexedselvariable}=\headindexof{\arbset}}\} \, ,
    \end{align}
    we have
    \begin{align*}
        \imsetof{\hlnstat,\meanset}{\facesymbol}
        = \imsetof{}{\arbset,\headindexof{\arbset}} \cap \imsetof{\hlnstat,\meanset}{} \, .
    \end{align*}
    Thus, we can choose $\imsetof{}{\arbset,\headindexof{\arbset}}$ for the representation of the face $\facesymbol$ (see \defref{def:faceRepresentability}) and further have that
    \begin{align*}
        \sum_{\imelementwith\in\imsetof{}{(\arbset,\headindexof{\arbset})}}\onehotmapofat{\imelementwith}{\headvariables}
        = \bigotimes_{\selindexin} \hardactlegwith \, ,
    \end{align*}
    where
    \begin{align*}
        \hardactlegwith
        = \begin{cases}
                \onehotmapofat{\headindexof{\selindex}}{\headvariableof{\selindex}} & \ifspace \selindex\in\arbset \\
                \onesat{\headvariableof{\selindex}} & \ifspace \selindex\notin\arbset \, .
        \end{cases}
    \end{align*}
    We have thus found an elementary activation tensor for the face $\facesymbol$.

    (ii)$\rightarrow$(i)
    Conversely, let $\acttensorwith$ be an elementary activation tensor of the face $\facesymbol$ in $\genmeanset$.
    Since $\acttensorwith$ is the sum of different one-hot encodings it is Boolean and we find an elementary decomposition $\acttensorwith=\bigotimes_{\selindexin}\acttensorlegwith$ such that the leg vectors $\acttensorwith$ are Boolean.
    Since $\facesymbol\neq\varnothing$ we further have for $\selindexin$ that $\acttensorlegwith\neq\zerosat{\headvariableof{\selindex}}$, and thus $\acttensorlegwith\in \{\fbasisat{\headvariableof{\selindex}},\tbasisat{\headvariableof{\selindex}},\onesat{\headvariableof{\selindex}}\}$
    We construct a set $\arbset\subset[\seldim]$
    \begin{align*}
        \arbset
        = \left\{\selindex \wcols {\acttensorlegwith}\neq\onesat{\headvariableof{\selindex}}\right\}
    \end{align*}
    and a tuple
    \begin{align*}
        \headindexof{\selindex}
        = \begin{cases}
              1 & \ifspace {\acttensorlegwith}=\onehotmapofat{1}{\headvariableof{\selindex}} \\
              0 & \ifspace {\acttensorlegwith}=\onehotmapofat{0}{\headvariableof{\selindex}}
        \end{cases} \, .
    \end{align*}
    By construction ${\acttensorlegwith} = \hardactlegwith$ (see \exaref{exa:hypercubeFaces}) follows for all $\selindexin$ and therefore $\acttensorwith=\hardacttensorwith$.
    We therefore have for the set \eqref{eq:cubeFaceVertices}
    \begin{align*}
        \acttensorwith
        = \sum_{\imelement\in\imsetof{}{\arbset,\headindexof{\arbset}}} \onehotmapofat{\imelement}{\headvariables}
    \end{align*}
    and since $\acttensorwith$ is an activation tensor for $\facesymbol$, that
    \begin{align*}
        \imsetof{\hlnstat,\meanset}{\facesymbol}
        = \imsetof{}{\arbset,\headindexof{\arbset}}\cap\imsetof{\hlnstat,\meanset}{} \, .
    \end{align*}
    Taking convex hulls on both sides, this is equivalent to
    \begin{align*}
        \facesymbol
        = \facesymbolof{(\arbset,\headindexof{\arbset})} \cap \genmeanset \, .
    \end{align*}
    We conclude that $\facesymbol$ is cube-like.
\end{proof}

\theref{the:faceMeasureHardLogicNetworks} implies that the maximum entropy distributions to hypercubes (see \exaref{exa:hypercubeFaces}) and simplices (see \exaref{exa:simplexFaces}) are in $\cansof{\hlnstat,\elgraph,\basemeasure}$, since these polytopes are cube-like.

We can further show, that in case of Boolean statistics any distribution in $\cansof{\hlnstat,\elgraph,\basemeasure}$ is a maximum entropy distribution.


\begin{theorem}
    Let $\hlnstat$ be a Boolean statistic.
    Any distribution in $\cansof{\hlnstat,\elgraph,\basemeasure}$ is a maximum entropy distribution with respect to $(\hlnstat,\meanparam,\basemeasure)$ where $\meanparam$ is its mean parameter.
\end{theorem}
\begin{proof}
    Let $\probtensor\in\cansof{\hlnstat,\elgraph,\basemeasure}$ be arbitrary.
    We then find an elementary tensor $\bigotimes_{\selenumeratorin}\hypercoreofat{\selindex}{\headvariableof{\selindex}}$ and $\partitionfunction\in\rr$ such that
    \begin{align*}
        \probwith
        = \frac{1}{\partitionfunction} \contractionof{\bencsstatwith,\bigotimes_{\selenumeratorin}\hypercoreofat{\selindex}{\headvariableof{\selindex}}}{\shortcatvariables} \, .
    \end{align*}
    We then define for each $\selindexin$ the vector $\kcoreofat{\selindex}{\headvariableof{\selindex}}$ as the support of $\hypercoreofat{\selindex}{\headvariableof{\selindex}}$ and
    \begin{align*}
        \weightparamat{\indexedselvariable}
        = \begin{cases}
              \lnof{\frac{\hypercoreofat{\selindex}{\headvariableof{\selindex}=1}}{\hypercoreofat{\selindex}{\headvariableof{\selindex}=0}}} & \ifspace \hypercoreofat{\selindex}{\headvariableof{\selindex}=0},\hypercoreofat{\selindex}{\headvariableof{\selindex}=1} \neq 0 \\
              0 \elsetext
        \end{cases} \, .
    \end{align*}
    Since $\hlnstat$ is a Boolean statistic the $\bigotimes_{\selindexin}\kcoreofat{\selindex}{\headvariableof{\selindex}}$ is a face activation tensor to the face
    \begin{align*}
        \facesymbol \coloneqq \hlnmeanset \cap \facesymbolof{(\arbset,\headindexof{\arbset})}
    \end{align*}
    where $\arbset=\{\selindex\wcols\kcoreofat{\selindex}{\headvariableof{\selindex}}\neq \onesat{\headvariableof{\selindex}}\}$ and for $\selindex\in\arbset$
    \begin{align*}
        \headindexof{\selindex}
        = \begin{cases}
              0 & \ifspace \kcoreofat{\selindex}{\headvariableof{\selindex}=1} = 0 \\
              1 & \elsetext
        \end{cases} \, .
    \end{align*}
    Therefore, $\probtensor$ is the maximum entropy distribution to the canonical parameters $\facesymbol$ and $\weightparam$.
\end{proof}

\subsection{Interpretation by propositional formulas}

Let us now assume that $\catdimof{\catenumerator}=2$ for any $\catenumeratorin$.
Following the formalism in \cite{goessmann_tensor_2026} we understand each coordinate $\sstatcoordinateof{\selindex}$ of the statistic as a propositional formula and the variables $\shortcatvariables$ as atoms of a propositional theory.
We use propositional connectives for syntactical representations of the formulas and denote by $\lnot^{1}$ the logical negation and by $\lnot^{0}$ the identity.
A Boolean tensor $\formulaat{\shortcatvariables}$ is called satisfiable, if there is at least one $\shortcatindicesin$ with $\formulaat{\indexedshortcatvariables}=1$.

\begin{theorem}\label{the:boolVecSat}
    Let $\sstat$ be a Boolean statistic.
    Any Boolean vector $\imelementwith\in\{0,1\}^{\seldim}$ is a vertex of $\hlnmeanset$ if and only if
    \begin{align*}
        \formulaofat{\imelement}{\shortcatvariables}
        \coloneqq \bigwedge_{\selindexin} \lnot^{1-\imelementat{\indexedselvariable}} \formulaofat{\selindex}{\shortcatvariables}
    \end{align*}
    is satisfiable.
\end{theorem}
\begin{proof}
    Let $\imelementwith\in\{0,1\}^{\seldim}$ be arbitrary.
    If and only if $\imelementwith$ is a vertex of $\hlnmeanset$ we find $\shortcatindicesin$ with $\imelementwith=\sencsstat{\indexedshortcatvariables,\selvariable}$.
    This is further equivalent to $\formulaofat{\selindex}{\indexedshortcatvariables}=\imelementat{\indexedselvariable}$ for any $\selindexin$, $\lnot^{1-\imelementat{\indexedselvariable}}\formulaofat{\selindex}{\indexedshortcatvariables}=1$ for any $\selindexin$, and $\formulaofat{\imelement}{\indexedshortcatvariables}=1$.
\end{proof}

Since the presence of Boolean vectors in $\hlnmeanset$ is related by \theref{the:boolVecSat} to the satisfiability of the conjunction $ \formulaofat{\imelement}{\shortcatvariables}$, the absence is related to entailment statements.
To be more precise, for any $\imelementwith\in\{0,1\}^{\seldim}$ and arbitrary $\arbset\subset[\seldim]$ we have $\imelementwith\notin\hlnmeanset$, if and only if
\begin{align*}
    \left(\bigwedge_{\selindex\in\arbset} \lnot^{1-\imelementat{\indexedselvariable}} \formulaofat{\selindex}{\shortcatvariables} \right)
    \models \left(\bigvee_{\selindex\notin\arbset}\lnot^{\imelementat{\indexedselvariable}} \formulaofat{\selindex}{\shortcatvariables} \right) \, .
\end{align*}
By definition (see \cite{russell_artificial_2021}) the latter holds if and only if
\begin{align*}
    \left(\bigwedge_{\selindex\in\arbset} \lnot^{1-\imelementat{\indexedselvariable}} \formulaofat{\selindex}{\shortcatvariables} \right)
    \land \lnot \left(\bigvee_{\selindex\notin\arbset}\lnot^{\imelementat{\indexedselvariable}} \formulaofat{\selindex}{\shortcatvariables} \right)
    = \formulaofat{\imelement}{\shortcatvariables}
\end{align*}
is not satisfiable.

Let us now relate two well-studied sets of formulas, namely atomic formulas and minterm formulas, with two well-studied polytopes, namely hypercubes and simplices.

\input{examples/boolean-stat/hypercube-atomic}

\input{examples/boolean-stat/simplex-minterm}

While hypercubes and simplices are cube-like, we now provide examples of statistics, which polytopes are not cube-like.

\input{examples/boolean-stat/nonel_hlnstat_maxent}

\subsection{Faces activated in the $\ttformat$ format}

\begin{figure}
    \input{tikz/boolean-stat/tt-format}
    \caption{
        $\ttformat$ format consisting of the $\ttformat$ hypergraph a) and decorating tensor networks b).
    }\label{fig:ttFormat}
\end{figure}

The $\ttformat$ format is a well-studied tensor network format with favorable numerical properties \cite{holtz_manifolds_2012} and applied across several scientific areas \cite{sander_large-scale_2025}.
This format describes tensor networks on a hypergraph (see \figref{fig:ttFormat})
\begin{align*}
    \ttformat
    = \big(
    \shortcatvariables\cup\decvariableof{[\catorder-1]},
    \{\{\decvariableof{\catenumerator-1},\catvariableof{\catenumerator},\decvariableof{\catenumerator}\}\wcols\catenumerator\in\{1,\ldots,\catorder-2\}\}
    \cup\{\{\catvariableof{0},\decvariableof{0}\},\{\decvariableof{\catorder-2},\catvariableof{\catorder-1}\}\}
    \big) \, ,
\end{align*}
where $\decvariableof{[\catorder-1]}$ are auxiliary hidden variables.
We now construct an example of a polytope with a face activatable with $\ttformat$, in which all hidden variables are of dimension $2$.

\input{examples/boolean-stat/tt-diagonal-face}

\subsection{Construction of Boolean statistics to polytopes}

We now show that any convex polytope with Boolean vertices in $\parspace$ is the mean polytope to a Boolean statistic.
To construct the corresponding formulas we denote by $\restrictionofto{\meanset}{\arbset}$ the restriction of the vectors in the set $\meanset$ to the coordinates in $\arbset$.

\begin{theorem}\label{the:boolStatToPolytope}
    Let $\meanset$ be an arbitrary polytope with Boolean vertices in $\parspace$.
    Then we construct propositional formulas on atoms $\catvariableof{[\seldim]}$ by
    \begin{align*}
        \formulaofat{0}{\catvariableof{0}} =
        \begin{cases}
            \top & \ifspace \restrictionofto{\meanset}{\{0\}} = \{1\} \\ 
            \bot & \ifspace \restrictionofto{\meanset}{\{0\}} = \{0\} \\
            \catvariableof{0} & \ifspace \restrictionofto{\meanset}{\{0\}} = [0,1] 
        \end{cases}
    \end{align*}
    and iteratively for $\selindexin$ with $\selindex\geq1$ by
    \begin{align*}
        \formulaofat{\selindex}{\catvariableof{[\selindex+1]}} =
        \bigwedge_{\imelementwith \in \restrictionofto{(\meanset\cap \{0,1\}^{\seldim})}{[\selindex]}}
            {\Big(}&
        \big(\bigwedge_{\secselindex\in[\selindex]} \lnot^{1-\imelementat{\selvariable=\secselindex}} \formulaofat{\secselindex}{\catvariableof{[\secselindex+1]}}\big) \Rightarrow \\
            &\begin{cases}
                \top & \ifspace \restrictionofto{(\meanset\cap\imelement\times\rr^{\seldim-\selindex})}{\{\selindex\}} = \{1\} \\
                \bot & \ifspace \restrictionofto{(\meanset\cap\imelement\times\rr^{\seldim-\selindex})}{\{\selindex\}} = \{0\} \\ 
                \catvariableof{\selindex} & \ifspace \restrictionofto{(\meanset\cap\imelement\times\rr^{\seldim-\selindex})}{\{\selindex\}} = [0,1] 
            \end{cases}
            {\Big)}.
    \end{align*}
    Here we denote by $\restrictionofto{\meanset}{V}$ the projections of the vertices in $\meanset$ onto the subspaces $V$, and by $0_{\seldim}$ the zero vector in $\rr^{\seldim}$.
\end{theorem}
\begin{proof}
    We show per induction, that for any $\selindexin$ the family of \HybridLogicNetworks{} with the statistic $\formulaof{[\selindex+1]}$ by the first $\selindex+1$ formulas has the mean polytope
    \begin{align}
        \label{eq:HLNindConTBS}
        \meansetof{\formulaof{[\selindex+1]},\trivbm}
        = \restrictionofto{\meanset}{[\selindex]} \, .
    \end{align}

    $"\selindex=0"$: The polytope $\restrictionofto{\meanset}{[\selindex]}=\{1\}$ (respectively $\restrictionofto{\meanset}{[\selindex]}=\{0\}$) is reproduced by $\formulaof{0}$ being a tautology (respectively a contradiction).
    In the case $\restrictionofto{\meanset}{[\selindex]}=[0,1]$ the polytope is reproduced by the any formula, which is neither a tautology nor a contradiction, and the atomic formula $\catvariableof{0}$ is an example of such a contingency.
    Since the projection of a 0-1 polytope onto the first coordinates is itself a 0-1 polytope, these are the only possible cases and we conclude that in all
    \begin{align*}
        \meansetof{\formulaof{[\selindex]},\trivbm}
        = \restrictionofto{\meanset}{[\selindex]} \, .
    \end{align*}

    $"\selindex\rightarrow\selindex+1"$:
    We define for any $\imelement\in\meansetof{\formulaof{[\selindex]},\trivbm}\cap\{0,1\}^{\selindex}$ the formula
    \begin{align*}
        \formulaofat{\selindex,\imelement}{\catvariableof{[\selindex+1]}}\coloneqq
        \bigwedge_{\secselindex\in[\selindex]} \lnot^{1-\imelementat{\selvariable=\secselindex}} \formulaofat{\secselindex}{\catvariableof{[\secselindex+1]}} \, .
    \end{align*}
    and notice that all formulas are satisfiable (since the corresponding vertices are in the $\meansetof{\formulaof{[\selindex]},\trivbm}$) and have disjoint models.
    We have for any Boolean vector $\imelement\in\{0,1\}^{\selindex}$ and $\lambda\in\{0,1\}$ that $\imelement\times\{\lambda\}\in\restrictionofto{\meanset}{[\selindex+1]}$, if and only if
    \begin{align*}
        \imelement \in \restrictionofto{\meanset}{[\selindex]} \quad \text{and} \quad
        \lambda \in \restrictionofto{\meanset\cap\imelement\times\rr^{\seldim-\selindex}}{[\selindex]} \, .
    \end{align*}
    With modus ponens on $\formulaofat{\selindex+1}{\catvariableof{[\selindex+2]}}$ this is further equivalent to the satisfaction of
    \begin{align*}
        \formulaofat{\selindex,\imelement}{\catvariableof{[\selindex+1]}} \land \lnot^{1-\lambda}\formulaofat{\selindex+1}{\catvariableof{[\selindex+2]}}
    \end{align*}
    and thus to $\imelement\times\{\lambda\}\in\meansetof{\formulaof{[\selindex+1]},\trivbm}$.
    The vertices of the 0/1 polytopes $\meansetof{\formulaof{[\selindex+1]},\trivbm}$ and $\restrictionofto{\meanset}{[\selindex+1]}$ thus coincide and both polytopes are therefore equivalent.

    By induction the equation \eqref{eq:HLNindConTBS} holds for arbitrary $\selindexin$.
    For $\selindex=\seldim-1$ this is the claim.
\end{proof}

\subsection{Examples in statistical physics and coding theory}

Let us now investigate cube-like faces to standard examples in statistical physics and coding theory.

\input{examples/boolean-stat/ising-two}

\input{examples/boolean-stat/ldpc-one}

%% file: examples/boolean-stat/hypercube-atomic.tex
\begin{example}[Atomic formulas]
    \label{exa:atomicFormulasHypercube}
    Let us consider the case of atomic formulas $\formulaofat{\selindex}{\shortcatvariables}$ defined with the coordinates at $\shortcatindicesin$ by
    \begin{align*}
        \formulaofat{\selindex}{\indexedshortcatvariables}
        = \catindexof{\selindex} \, .
    \end{align*}
    The mean polytope in this case is the $\catorder$-dimensional hypercube (see \exaref{exa:hypercubeFaces})
    \begin{align*}
        \meansetof{\atomformulaset,\ones}
        = \fullparcube
    \end{align*}
    which is called a simple polytope, since each vertex is contained in the minimal number of $\catorder$ facets.

    Each face is characterized by the projections onto each variable, which is either $\{0\}$, $\{1\}$ or $[0,1]$.
    The projections are represented by the tuple $\hardparam$ defined in the following way:
    \begin{itemize}
        \item We define the set $\hardlegset\subset[\atomorder]$ of variables, such that the projection onto the variable is $\{0\}$ or $\{1\}$
        \item We define to each $\selindex\in\hardlegset$ an index $\headindexof{\selindex}=0$ if the projection is $\{0\}$ and $\headindexof{\selindex}=1$ if the projection is $\{1\}$.
    \end{itemize}
    Trivially, each face of the hypercube is a cube face. 

    More generally, the mean polytope to a collection $\formulaofat{\selindex}{\shortcatvariables}$ of propositional formulas is the hypercube, if and only if for any Boolean vector $\imelementwith$ we have
    \begin{align*}
        \contraction{\bigwedge_{\selindexin}\lnot^{1-\imelementat{\indexedselvariable}}\formulaofat{\selindex}{\shortcatvariables}}
        \neq 0 \, .
    \end{align*}
    This means that no conjunction of possibly negated formulas from this statistic entails nor contradicts another formula.

\end{example}

%% file: examples/boolean-stat/simplex-minterm.tex
\begin{example}[Minterm formulas build simplices]
    \label{exa:mintermHLNSet}
    The set of minterm formulas is indexed by $\imelementwith\in\{0,1\}^{\seldim}$ and given by
    \begin{align*}
        \formulaofat{\imelement}{\shortcatvariables}
        = \bigwedge_{\selindex\in[\catorder]} \lnot^{1-\imelementat{\indexedselvariable}} \, \catvariableof{\selindex}
        = \onehotmapofat{\imelementwith}{\shortcatvariables}
    \end{align*}
    where by $\catvariableof{\selindex}$ we denote the $\selindex$-th atomic formula (see \exaref{exa:atomicFormulasHypercube}).
    The mean polytope to the minterm statistic and base measure $\basemeasure$ is the standard simplex of dimension
    \begin{align*}
        \cardof{\left\{\imelementwith\wcols\imelement\in\{0,1\}^{\seldim}\ncond\basemeasureat{\shortcatvariables=\imelementwith}\neq 0\right\}}-1 \, ,
    \end{align*}
    that is the set
    \begin{align*}
        \meansetof{\mintermformulaset,\basemeasure}
        = \convhullof{\onehotmapofat{\imelementwith}{\headvariables}\wcols\imelement\in\{0,1\}^{\seldim}\ncond\basemeasureat{\shortcatvariables=\imelementwith}\neq 0} \, .
    \end{align*}
    We notice that for Boolean base measures $\meansetof{\mintermformulaset,\basemeasure}=\realizabledistsof{\basemeasure}$.
    The simplex is further cube-like, since each face is indexed by a subset $\arbset\subset\{0,1\}^{\seldim}$ of vertices and the intersection of the simplex with the cube face
    $\facesymbolof{(\{0,1\}^p\backslash\arbset,\onetuple{\{0,1\}^p\backslash\arbset})}$.




    %
    More generally, the mean polytope to a Boolean statistic is a standard simplex, if and only if each target coordinate restriction $\formulaof{\secselindex}$ contradicts all others, that is
    \begin{align*}
        \uniquantwrtof{\selindex\neq\secselindex\in[\seldim]}{\formulaof{\selindex}\models\lnot\formulaof{\secselindex}} \, .
    \end{align*}
\end{example}

%% file: examples/boolean-stat/nonel_hlnstat_maxent.tex
\begin{example}[Maximum entropy distribution with non-elementary activation cores]\label{exa:nonelHlnstat}

    Consider two atomic variables $\catvariableof{0}$ and $\catvariableof{1}$ and a statistic $\formulaset$ consisting of the formulas
    \begin{align*}
        \formulaofat{0}{\catvariableof{0},\catvariableof{1}} = \left( \catvariableof{0} \land \catvariableof{1} \right) \quad, \quad
        \formulaofat{1}{\catvariableof{0},\catvariableof{1}} = \left( \catvariableof{0} \Rightarrow \catvariableof{1} \right)
    \end{align*}
    with the coordinatewise expressions
    \begin{align*}
        \formulaofat{0}{\catvariableof{0},\catvariableof{1}} =
        \coloredmatrixof{
            0 & 0 \\
            0 & 1
        }{\catvariableof{0},\catvariableof{1}}
        \quad, \quad
        \formulaofat{1}{\catvariableof{0},\catvariableof{1}} =
        \coloredmatrixof{
            1 & 1 \\
            0 & 1
        }{\catvariableof{0},\catvariableof{1}} \, .
    \end{align*}
    We can think of $\catvariableof{0}$ as a feature on an invoice, and $\catvariableof{1}$ as a feature on the accounting proposal.

    From this we have
    \begin{align*}
        \bencodingofat{(\formulaof{0},\formulaof{1})}{\headvariableof{0}=0,\headvariableof{1}=0,\catvariableof{0},\catvariableof{1}} &=
        \coloredmatrixof{
            0 & 0 \\
            1 & 0
        }{\catvariableof{0},\catvariableof{1}} \quad, \quad \\
       \bencodingofat{(\formulaof{0},\formulaof{1})}{\headvariableof{0}=0,\headvariableof{1}=1,\catvariableof{0},\catvariableof{1}} &=
        \coloredmatrixof{
            1 & 1 \\
            0 & 0
        }{\catvariableof{0},\catvariableof{1}} \quad, \quad \\
        \bencodingofat{(\formulaof{0},\formulaof{1})}{\headvariableof{0}=1,\headvariableof{1}=0,\catvariableof{0},\catvariableof{1}} &=
        \coloredmatrixof{
            0 & 0 \\
            0 & 0
        }{\catvariableof{0},\catvariableof{1}} \quad \text{and} \quad \\
        \bencodingofat{(\formulaof{0},\formulaof{1})}{\headvariableof{0}=1,\headvariableof{1}=1,\catvariableof{0},\catvariableof{1}} &=
        \coloredmatrixof{
            0 & 0 \\
            0 & 1
        }{\catvariableof{0},\catvariableof{1}} \, .
    \end{align*}

    Since the only vanishing slice of $\bencodingof{\formulaset}$ with respect to the head variables is that to $\headindexof{0,1} = (1,0)$, the vertices of the mean polytope are the vectors to the other head indices.
    The mean polytope is the convex hull of these vertices
    \begin{align*}
        \meansetof{(\formulaof{0},\formulaof{1})} =
        \convhullof{\coloredmatrixof{
                        0 \\ 0
        }{\selvariable},
            \coloredmatrixof{
                0 \\ 1
            }{\selvariable},
            \coloredmatrixof{
                1 \\ 1
            }{\selvariable}} \, .
    \end{align*}

    This polytope has a non cube-like face (sketched blue in \figref{fig:nonelHlnstatMaxent}), which is the convex hull of the vertices $(0 \, 0)^T, \, (1 \, 1)^T$.
    This face is parametrized by the ($\cpformat$-rank 2) hard activation core
    \begin{align*}
        \kcoreofat{(0,0),(1,1)}{\headvariableof{0},\headvariableof{1}} =
        \onehotmapofat{(0,0)}{\headvariableof{0},\headvariableof{1}}
        + \onehotmapofat{(1,1)}{\headvariableof{0},\headvariableof{1}} =
        \coloredmatrixof{
            1 & 0 \\
            0 & 1
        }{\headvariableof{0},\headvariableof{1}}
    \end{align*}
    and has the refined base measure
    \begin{align*}
        \contractionof{\kcoreofat{(0,0),(1,1)}{\headvariableof{0},\headvariableof{1}}
            ,\bencodingofat{\formulaset}{\headvariableof{0},\headvariableof{1},\catvariableof{0},\catvariableof{1}}}{\catvariableof{0},\catvariableof{1}}
        =  \coloredmatrixof{
            0 & 0 \\
            1 & 1
        }{\catvariableof{0},\catvariableof{1}}  \, .
    \end{align*}
    Any mean parameter $\meanparam$ on the interior of that face can be parametrized by a scalar $\lambda\in(0,1)$
    \begin{align*}
        \meanparamofat{\lambda}{\selvariable} =
        \coloredmatrixof{\lambda & \lambda}{\selvariable} \, .
    \end{align*}
    With the canonical parameters $\canparamat{\selvariable}\in\rr^2$ of the maximum entropy distributions on this face by
    \begin{align*}
        \probat{\catvariableof{0},\catvariableof{1}} =
        \frac{1}{1+\expof{\canparamat{\selvariable=0}+\canparamat{\selvariable=1}}}
        \coloredmatrixof{
            0 & 0                                                               \\
            1 & \expof{\canparamat{\selvariable=0}+\canparamat{\selvariable=1}}
        }{\catvariableof{0},\catvariableof{1}}
    \end{align*}
    we get the correspondence by the sigmoid
    \begin{align*}
        \lambda = \frac{1}{1+\expof{-(\canparamat{\selvariable=0}+\canparamat{\selvariable=1})}} \, .
    \end{align*}

    Note, that the hard activation core $\kcoreofat{(0,0),(1,1)}{\headvariableof{0},\headvariableof{1}}$ to the blue face is the only non-elementary activation core.
    While the vertices have always elementary cores, the other non-vertex faces have elementary activation cores
    \begin{align*}
        \kcoreofat{(0,0),(1,0),(1,1)}{\headvariableof{0},\headvariableof{1}}
        &= \coloredmatrixof{
            1 & 1 \\
            1 & 1
        }{\headvariableof{0},\headvariableof{1}}
        = \onesat{\headvariableof{0}} \otimes \onesat{\headvariableof{1}}
        , \quad\\
        \kcoreofat{(0,0),(1,0)}{\headvariableof{0},\headvariableof{1}}
        &= \coloredmatrixof{
            1 & 0 \\
            1 & 0
        }{\headvariableof{0},\headvariableof{1}}
        = \onesat{\headvariableof{0}} \otimes \onehotmapofat{0}{\headvariableof{1}}
        ,\\
        \kcoreofat{(1,0),(1,1)}{\headvariableof{0},\headvariableof{1}}
        &= \coloredmatrixof{
            1 & 1 \\
            0 & 0
        }{\headvariableof{0},\headvariableof{1}}
        = \onehotmapofat{0}{\headvariableof{0}} \otimes \onesat{\headvariableof{1}}   \, .
    \end{align*}
    The maximum entropy distributions to mean parameters on the interior of all other faces than the blue face are represented by \ComputationActivationNetwork{}s with only elementary activation cores.
    The corresponding lattice is sketched in \figref{fig:latticeNonelHlnstat}.

    \begin{figure}
        \begin{center}
            \input{tikz/boolean-stat/nonel_hlnstat_maxent.tikz}
        \end{center}
        \caption{
            a) Mean polytope of the statistic $\hlnstat=(\catvariableof{0} \land \catvariableof{1}, \catvariableof{0} \Rightarrow \catvariableof{1})$ (thick), as a subset of the cube $[0,1]^2$ (dashed).
            The blue line is the face of the polytope, which is not cube-like, that is not an intersection of the polytope with the faces of the hypercube.
            We further define for $\lambda\in(0,1)$ a mean parameter $\meanparamofat{\lambda}{\selvariable} = [\lambda \,  \lambda]^T$ which is on the interior of the blue face.
            b) Corresponding \ComputationActivationNetwork{} being the maximum entropy distribution reproducing $\meanparamofat{\lambda}{\selvariable}$, when $\lambda$ is the sigmoid of $\canparamat{\selvariable=0}+\canparamat{\selvariable=1}$.
        }\label{fig:nonelHlnstatMaxent}
    \end{figure}

    \begin{figure}
        \begin{center}
            \input{tikz/boolean-stat/lattice_nonel_hlnstat.tikz}
        \end{center}
        \caption{
            Face lattice $\facelatticeof{(\catvariableof{0} \land \catvariableof{1},\catvariableof{0} \Rightarrow \catvariableof{1})}$ to \exaref{exa:nonelHlnstat}.
            The directed arrows represent inclusion of the faces, which is a partial order of the faces.
            The face \textcolor{\concolor}{$\facesymbolof{(0,0),(1,1)}$} is the only face, which is not representable by an elementary hypergraph.
            It is representable in a $\cpformat$ with hidden rank $2$
        }\label{fig:latticeNonelHlnstat}
    \end{figure}

\end{example}

%% file: tikz/boolean-stat/nonel_hlnstat_maxent.tikz
\begin{tikzpicture}[scale=0.35]

    \node[anchor=center] at (-4,6) {$a)$};

    \node[anchor=east] at (0,0) {$\begin{pmatrix}
                                      0 \\ 0
    \end{pmatrix}$};
    \node[anchor=west] at (5,0) {$\begin{pmatrix}
                                      1 \\ 0
    \end{pmatrix}$};
    \node[anchor=west] at (5,5) {$\begin{pmatrix}
                                      1 \\ 1
    \end{pmatrix}$};
    \node[anchor=east] at (0,5) {$\begin{pmatrix}
                                      0 \\ 1
    \end{pmatrix}$};

    \drawvectormark{0}{0}
    \drawvectormark{0}{5}
    \drawvectormark{5}{0}
    \drawvectormark{5}{5}

    \draw[thick] (5,5) -- (5,0) -- (0,0);
    \draw[dashed] (0,0) -- (0,5) -- (5,5);
    \draw[\concolor, thick] (0,0) -- (5,5);

    \drawvectormark{3}{3}
    \node[anchor=east] at (3,3) {$\meanparam^{\lambda}$};

    \begin{scope}
        [shift={(20,2)}]

        \node[anchor=center] at (-4,4) {$b)$};




        \draw[\concolor] (-1,1) rectangle (5,4);
        \node[anchor=center,\concolor] (A) at (2,2.5) {\corelabelsize $\begin{pmatrix}
                                        1 & 0 \\
                                        0 & 1
        \end{pmatrix}$};

        \draw[->-] (0,-1)--(0,0);
        \node[right] (text) at (0,0) {\colorlabelsize $\headvariableof{0}$};
        \draw[\concolor] (0,0)--(0,1);
        \drawvariabledot{0}{0}

        \draw[\probcolor] (0,0) -- (-1.5,0);
        \draw[\probcolor] (-1.5,1.5) rectangle (-8.5,-1.5);
        \node[anchor=center,\probcolor] (text) at (-5,0) {\corelabelsize $\begin{pmatrix}
                                                                             1 \\
                                                                             \expof{\canparamat{\selvariable=0}}
        \end{pmatrix}$};

        \draw[->-] (4,-1)--(4,0);
        \node[left] (text) at (4,0) {\colorlabelsize $\headvariableof{1}$};
        \draw[\concolor] (4,0)--(4,1);
        \drawvariabledot{4}{0}

        \draw[\probcolor] (4,0) -- (5.5,0);
        \draw[\probcolor] (5.5,1.5) rectangle (12.5,-1.5);
        \node[anchor=center,\probcolor] (text) at (9,0) {\corelabelsize $\begin{pmatrix}
                                                                             1 \\
                                                                             \expof{\canparamat{\selvariable=1}}
        \end{pmatrix}$};


        \draw (-1,-1) rectangle (5,-3);
        \node[anchor=center] (text) at (2,-2) {\corelabelsize $\bencodingof{(\formulaof{0},\formulaof{1})}$};
        \draw[-<-] (0,-3)--(0,-5) node[midway,left] {\colorlabelsize $\catvariableof{0}$};

        \draw[-<-] (4,-3)--(4,-5) node[midway,right] {\colorlabelsize $\catvariableof{1}$};

    \end{scope}

\end{tikzpicture}

%% file: tikz/boolean-stat/lattice_nonel_hlnstat.tikz
\begin{tikzpicture}[scale=0.35]

    \node[anchor=center] at (0,9) {$\facesymbolof{\varnothing}$};
    \draw[->] (-1,8.5) -- (-6,6.5);
    \draw[->] (0,8) -- (0,7);
    \draw[->] (1,8.5) -- (6,6.5);

    \node[anchor=center] at (-8,6) {$\facesymbolof{(0,0)}$};
    \draw[->] (-8,5) -- (-8,4);

    \node[anchor=center] at (0,6) {$\facesymbolof{(1,0)}$};
    \draw[->] (-1,5.25) -- (-6,3.5);
    \draw[->] (0,5) -- (0,4);

    \node[anchor=center] at (8,6) {$\facesymbolof{(1,1)}$};

    \draw[->] (7,5.25) -- (2,3.5);

    \node[anchor=center] at (-8,3) {$\facesymbolof{(0,0),(1,0)}$};
    \draw[->] (-7.5,2) -- (-1.5,1);
    \node[anchor=center] at (0,3) {$\facesymbolof{(1,0),(1,1)}$};
    \draw[->] (0,2) -- (0,1);

    \node[] (shift) at (5,0) {};
    \draw[->] (8,5) -- ($(7,4)+(shift)$);
    \draw[->] (-7.5,5) -- ($(6,3.5)+(shift)$);
    \node[anchor=center] at ($(8,3)+(shift)$) {\textcolor{\concolor}{$\facesymbolof{(0,0),(1,1)}$}};
    \draw[->] ($(7,2)+(shift)$) -- (1.5,1);

    \draw[dashed] (9.5,-1) -- (9.5,10);
    \node[anchor=west] at (9.75,9.5) {\corelabelsize $\cprankof{\facesymbol}=2$};
    \node[anchor=east] at (9.25,9.5) {\corelabelsize $\cprankof{\facesymbol}=1$};

    \node[anchor=center] at (0,0) {$\facesymbolof{(0,0),(1,0),(1,1)}$};

\end{tikzpicture}

%% file: tikz/boolean-stat/tt-format.tex
\stantikzpic{

    \drawtextnode{0}{5}{$a)$}

    \drawvariablenode{0}{0}{$\catvariableof{0}$}{X0}
    \drawvariablenode{4}{0}{$\catvariableof{1}$}{X1}
    \drawvariablenode{8}{0}{$\catvariableof{2}$}{X2}
    \drawtextnode{11}{0}{$\cdots$}
    \drawvariablenode{14}{0}{$\catvariableof{\catorder-1}$}{Xdmin1}

    \drawvariablenode{2}{4}{$\decvariableof{0}$}{I0}
    \drawvariablenode{6}{4}{$\decvariableof{1}$}{I1}
    \drawtextnode{9}{4}{$\cdots$}
    \drawvariablenode{12}{4}{$\decvariableof{\catorder-2}$}{Idmin2}

    \node[anchor=south] (E0) at (0.5,2) {\colorlabelsize $e_0$};
    \draw (X0) -- (I0);

    \node[anchor=south] (E1) at (4,2) {\colorlabelsize $e_1$};
    \draw (4,2) -- (X1);
    \draw (4,2) -- (I0);
    \draw (4,2) -- (I1);

    \node[anchor=south] (E1) at (8,2) {\colorlabelsize $e_2$};
    \draw (8,2) -- (X2);
    \draw (8,2) -- (I1);
    \draw (8,2) -- (9,3);

    \node[anchor=south] (E0) at (13.5,2) {\colorlabelsize $e_{\catorder\shortminus2}$};
    \draw (Xdmin1) -- (Idmin2);

    \begin{scope}[shift={(18,0)}]
        \drawtextnode{-2}{5}{$b)$}

        \drawtensorblock{0}{2.5}{1}{$\hypercoreof{0}$}
        \drawvwire{0}{1.5}{$\catvariableof{0}$}
        \drawtensorblock{4}{2.5}{1}{$\hypercoreof{1}$}
        \drawvwire{4}{1.5}{$\catvariableof{1}$}

        \drawtensorblock{8}{2.5}{1}{$\hypercoreof{2}$}
        \drawvwire{8}{1.5}{$\catvariableof{2}$}

        \drawtextnode{12}{2.5}{$\cdots$}

        \drawtensorblock{16}{2.5}{1}{$\hypercoreof{\catorder\shortminus1}$}
        \drawvwire{16}{1.5}{$\catvariableof{\catorder\shortminus1}$}

        \drawhwire{1}{2.5}{$\decvariableof{0}$}
        \drawhwire{5}{2.5}{$\decvariableof{1}$}
        \drawhwire{9}{2.5}{$\decvariableof{2}$}
        \drawhwire{13}{2.5}{$\decvariableof{\catorder\shortminus1}$}
    \end{scope}

}

%% file: examples/boolean-stat/tt-diagonal-face.tex
\begin{example}[$\ttformat$ representation of diagonal faces]
    \label{exa:tt_diagonal_face}
    Consider the vertex set
    \begin{align*}
        \imset
        \coloneqq \{0,1\}^{\seldim} / \imelementat{\selvariable}
    \end{align*}
    for some Boolean vector $\imelementat{\selvariable}$.
    We are interested in the face $\facesymbolof{\triangleleft,1}$ with the normal $\canparamwith=\frac{1}{2} (2\imelementwith - \onesat{\selvariable})$.
    Its vertices are the $\seldim$ elements of $\{0,1\}^{\seldim}$, which differ from $\imelementwith$ in exactly one coordinate (i.e. those with Hamming distance of $1$ from $\imelementat{\selvariable}$).
    We denote these vertices by $\imelementof{\selindex}$ for $\selindexin$, which coordinates to $\secselindex\in[\seldim]$ are
    \begin{align*}
        \imelementofat{\selindex}{\selvariable=\secselindex}
        = \begin{cases}
              \imelementat{\selvariable=\secselindex} & \ifspace \secselindex\neq\selindex \\
              1-\imelementat{\selvariable=\secselindex} & \ifspace \secselindex=\selindex
        \end{cases} \, .
    \end{align*}

    We now represent their sum in an $\ttformat$ format.
    The Boolean hidden variables are denoted by $\decvariableof{[\seldim-1]}$ and are interpreted as indicators, whether the coordinate flip has happened in $[\selindex]$ coordinates.
    We construct a $\ttformat$ core for $\selindex=0$ by
    \begin{align*}
        \hypercoreofat{0}{\headvariableof{0},\decvariableof{0}}
        = \onehotmapofat{\imelementat{\selvariable=0}}{\headvariableof{0}} \otimes \fbasisat{\decvariableof{0}}
        + \onehotmapofat{1-\imelementat{\selvariable=0}}{\headvariableof{0}} \otimes \tbasisat{\decvariableof{0}} \, ,
    \end{align*}
    for $\selindex\notin\{0,\seldim-1\}$ the cores
    \begin{align*}
        \hypercoreofat{\selindex}{\decvariableof{\selindex-1},\headvariableof{\selindex},\decvariableof{\selindex}}
        = \, & \fbasisat{\decvariableof{\selindex-1}} \otimes \onehotmapofat{1-\imelementat{\indexedselvariable}}{\headvariableof{\selindex}} \otimes \tbasisat{\decvariableof{\selindex}}
        + \fbasisat{\decvariableof{\selindex-1}} \otimes \onehotmapofat{\imelementat{\indexedselvariable}}{\headvariableof{\selindex}} \otimes \fbasisat{\decvariableof{\selindex}} \\
        &+ \tbasisat{\decvariableof{\selindex-1}} \otimes \onehotmapofat{\imelementat{\indexedselvariable}}{\headvariableof{\selindex}} \otimes \tbasisat{\decvariableof{\selindex}}
    \end{align*}
    and for $\selindex=\seldim-1$
    \begin{align*}
        \hypercoreofat{\seldim-1}{\decvariableof{\seldim-2},\headvariableof{\seldim-1}}
        = \fbasisat{\decvariableof{\seldim-2}} \otimes \onehotmapofat{1-\imelementat{\selvariable=\seldim-1}}{\headvariableof{\seldim-1}}
        + \tbasisat{\decvariableof{\seldim-2}} \otimes \onehotmapofat{\imelementat{\selvariable=\seldim-1}}{\headvariableof{\seldim-1}} \, .
    \end{align*}
    For this tensor network in the $\ttformat$ format we then have
    \begin{align*}
        \sum_{\selindexin} \onehotmapofat{\imelementof{\selindex}}{\headvariables}
        = \contractionof{\{\hypercoreofat{0}{\headvariableof{0},\decvariableof{0}},\hypercoreofat{\seldim-1}{\decvariableof{\seldim-2},\headvariableof{\seldim-1}}\}\cup\{\hypercoreofat{\selindex}{\decvariableof{\selindex-1},\headvariableof{\selindex},\decvariableof{\selindex}}\wcols\selindex\in\{1,\ldots,\seldim-2\}\}}{\headvariables} \, .
    \end{align*}
    The dimension $2$ to the hidden variables is furthermore minimal, since each matrification based on a partition of $[\seldim]$ into non-empty sets has a matrix rank of $2$ \cite{holtz_manifolds_2012}.
\end{example}

%% file: examples/boolean-stat/ising-two.tex
\begin{example}[Ising model on $2$ nodes]
    Consider two Boolean variables $\catvariableof{0},\catvariableof{1}$ and the Ising statistic $\sstat=(\formulaof{0},\formulaof{1},\formulaof{2})$ by three propositional formulas
    \begin{align*}
        \formulaofat{0}{\catvariableof{[2]}} = \catvariableof{0} \quad , \quad
        \formulaofat{1}{\catvariableof{[2]}} = \catvariableof{1} \andspace
        \formulaofat{2}{\catvariableof{[2]}} = \catvariableof{0}\land\catvariableof{1} \, .
    \end{align*}
    In the Ising interpretation, the Boolean variables represent interacting spins at two locations.
    Their value is then measured by the first two formulas and their interaction by the third.

    The vertices of the mean polytope to this statistic are
    \begin{align*}
        &\sencsstatat{\catvariableof{[2]}=(0,0),\selvariable} =
        \coloredmatrixof{0 \\ 0 \\ 0}{\selvariable}
        \quad , \quad
        \sencsstatat{\catvariableof{[2]}=(0,1),\selvariable} =
        \coloredmatrixof{0 \\ 1 \\ 0}{\selvariable}
        \quad , \quad \\
        &\sencsstatat{\catvariableof{[2]}=(1,0),\selvariable} =
        \coloredmatrixof{1 \\ 0 \\ 0}{\selvariable}
        \quad \text{and} \quad
        \sencsstatat{\catvariableof{[2]}=(1,1),\selvariable} =
        \coloredmatrixof{1 \\ 1 \\ 1}{\selvariable}
        \quad .
    \end{align*}

    The mean polytope is the convex hull of these, sketched as:
    \begin{center}
        \tdplotsetmaincoords{55}{20} 

        \begin{tikzpicture}[tdplot_main_coords, scale=2,
            mainline/.style={thick},
            invisibleline/.style={dashed, gray}
        ]

            \coordinate (A) at (0, 0, 0);
            \coordinate (B) at (0, 1, 0);
            \coordinate (C) at (1, 0, 0);
            \coordinate (D) at (1, 1, 1);

            \fill[cyan!20, opacity=0.6] (A) -- (B) -- (C) -- cycle;

            \draw[mainline] (B) -- (C) -- (D) -- cycle; 
            \draw[mainline] (A) -- (B);
            \draw[mainline] (A) -- (C);
            \draw[mainline] (A) -- (D);

            \draw[invisibleline] (A) -- (B);
            \draw[invisibleline] (A) -- (C);
            \draw[invisibleline] (A) -- (D);

            \foreach \point/\label/\pos in {A/000/below left, B/010/below right, C/100/above left, D/111/above} {
                \draw[fill=black] (\point) circle (1pt) node[\pos] {$\label$};
            }

            \draw[->, gray] (0,0,0) -- (1.2,0,0) node[right] {$\meanparamat{\selvariable=0}$};
            \draw[->, gray] (0,0,0) -- (0,1.2,0) node[above] {$\meanparamat{\selvariable=1}$};
            \draw[->, gray] (0,0,0) -- (0,0,1.2) node[above] {$\meanparamat{\selvariable=2}$};

        \end{tikzpicture}
    \end{center}
    We notice, that the convex hull of any three out of the four vertices builds a facet.
    The only cube-like facet of these four (sketched in cyan in the above plot) is the convex hull of
    \begin{align*}
        \convhullof{\coloredmatrixof{0 \\ 0 \\ 0}{\selvariable},\coloredmatrixof{1 \\ 0 \\ 0}{\selvariable},\coloredmatrixof{0 \\ 1 \\ 0}{\selvariable}}
    \end{align*}
    to which we have the parametrization by $(\hardlegset,\headindexof{\hardlegset})=(\{2\},0)$.

    As a consequence, there are maximum entropy distributions of mean parameters in the Ising statistics, which do not have a representation by an elementary \CompActNet{} in the Ising statistic.
\end{example}

%% file: examples/boolean-stat/ldpc-one.tex
\begin{example}[Codeword Polytopes]
    Consider the codeword polytope with the vertices
    \begin{align*}
        \imset
        =\left\{\headindexof{[3]} \wcols \sum_{\catenumerator\in[3]} \headindexof{\catenumerator} \mod 2=0 \right\}
        = \left\{
              \coloredmatrixof{0 \\ 0 \\ 0}{\selvariable},
              \coloredmatrixof{0 \\ 1 \\ 1}{\selvariable},
              \coloredmatrixof{1 \\ 0 \\ 1}{\selvariable},
              \coloredmatrixof{1 \\ 1 \\ 0}{\selvariable}
        \right\}
    \end{align*}
    We interpret the vertices by the strings of length $3$, which parity vanishes.
    Such constraints are common in the construction of codewords using parity-check codes (see \cite{gallager_low-density_1963}).
    \begin{center}
        \tdplotsetmaincoords{55}{30} 

        \begin{tikzpicture}[tdplot_main_coords, scale=2,
            mainline/.style={thick},
            invisibleline/.style={dashed, gray}
        ]

            \coordinate (A) at (0, 0, 0);
            \coordinate (B) at (0, 1, 1);
            \coordinate (C) at (1, 0, 1);
            \coordinate (D) at (1, 1, 0);


            \draw[mainline] (B) -- (C) -- (D) -- cycle; 
            \draw[mainline] (A) -- (B);
            \draw[mainline] (A) -- (C);
            \draw[mainline] (A) -- (D);

            \draw[invisibleline] (A) -- (B);
            \draw[invisibleline] (A) -- (C);
            \draw[invisibleline] (A) -- (D);

            \foreach \point/\label/\pos in {A/000/below left, B/011/below right, C/101/above left, D/110/above} {
                \draw[fill=black] (\point) circle (1pt) node[\pos] {$\label$};
            }

            \draw[->, gray] (0,0,0) -- (1.2,0,0) node[right] {$\meanparamat{\selvariable=0}$};
            \draw[->, gray] (0,0,0) -- (0,1.2,0) node[above] {$\meanparamat{\selvariable=1}$};
            \draw[->, gray] (0,0,0) -- (0,0,1.2) node[above] {$\meanparamat{\selvariable=2}$};

        \end{tikzpicture}
    \end{center}

    All facets are non-cube-like, whereas all other faces are cube-like (the edges are labeled by the $6$ facets of $[0,1]^3$).

    Each facet is characterized by the Hamming distance of exactly $1$ to a single of the $4$ cube vertices, which are not codewords (see Example~3.9 in \cite{wainwright_graphical_2008}).
    We can construct $\ttformat$ activation tensors with rank $2$, which calculate the Hamming distance to these vertices via hidden ranks.
    Each core tensor in that format adds the contribution of a mode to the Hamming distance.

    In more generality, Low-Density-Parity-Checking Codes (see \cite{gallager_low-density_1963}) construct codewords by a collection of parity constraints on subsets of variables.
    Such polytopes are also known as cycle polytopes, and studied e.g. in \cite{grotschel_geometric_1993}.

\end{example}

%% file: sections/outlook.tex
\section{Outlook}

This work characterized all discrete maximum entropy distributions as \ComputationActivationNetworks{}, a class of tensor networks.
The mean polytope contains those mean parameters, to which entropy maximization problems are feasible.
While exponential families are known to be the maximum entropy distributions to the interior of the polytope, we showed that refining the base measure with respect to faces enables the representation of maximum entropy distributions to mean parameters on the particular face.
To this end we utilized the tensor network formalism and showed how to represent exponential families and their generalizations to maximum entropy distributions.

One important application is estimation of maximum entropy distributions from data, such as learning in exponential families and graphical models \cite{jordan_learning_1998}.
Maximum likelihood problems are known to be dual to maximum entropy problems \cite{koller_probabilistic_2009} and therefore face similar representation challenges.
Here the mean parameters of empirical distributions serve as sufficient statistics and are matched by trainable weights in algorithms such as iterative proportional fitting \cite{darroch_generalized_1972,csiszar_geometric_1989}.
Optimizing parameters within a poor choice of learning architecture will lead to numerical instabilities, resulting from the absence of an optimum.
Choosing the right tensor format representing the refined base measure guarantees the existence of an optimum.

Further research might derive a relation between tensor format complexity \cite{hackbusch_tensor_2012} to notions of facet complexity \cite{grotschel_geometric_1993}.
The complexity of faces from randomly drawn polytopes is further interesting, especially in the asymptotic limit of large statistics \cite{vershynin_high-dimensional_2018, wainwright_high-dimensional_2019}.

%
%
%